\documentclass[11pt]{article} 

\usepackage{preamble} 

\usepackage{tipa} 

\usepackage{linguex} 

\usepackage{graphicx} 

\usepackage{float} 

\usepackage{amsmath, amssymb} 

\usepackage{amsfonts} 

\usepackage{tikz-cd} 

\usepackage{mathabx} 

\usepackage{enumitem} 

\usepackage{capt-of} 

\usepackage[normalem]{ulem} 

\usepackage{mathtools} 

\newcommand{\NE}{\normalfont\textsc{ne}} 

\newcommand{\Sup}{\langle \sup\rangle } 

\newcommand{\Inf}{\langle \inf\rangle } 

\newcommand{\Sast}{\langle {\sup}^\ast\rangle } 

\newcommand{\Past}{ \textnormal{\textsc{P}} } 

\newcommand{\PastH}{ \textnormal{\textsc{H}} } 

\newcommand{\disji}{%
  \mathrel{\raisebox{0.3pt}{%
    \rotatebox[origin=c]{-90}{\scalebox{0.75}{$\geq$}}}}} 
\newcommand{\vvee}{\disji} 

\title{Teams, Truthmakers, Translations
} 

\author{Maria Aloni \\ 
    \small University of Amsterdam \\
        \Hor 
    Alexandru Baltag \\
    \small University of Amsterdam \\
        \Hor
    Johan van Benthem \\
    \small University of Amsterdam \\
    \small Stanford University \\
    \small Tsinghua University \\
        \Ver 
    Nick Bezhanishvili \\ 
    \small University of Amsterdam \\
        \Hor
    Søren Brinck Knudstorp \\ 
    \small University of Amsterdam \\
}

\begin{document}

\setheaderinfo{Teams, Truthmakers, Translations
}{Aloni, Baltag, van Benthem, Bezhanishvili, Knudstorp} 

\maketitle 
\thispagestyle{ACtitlepageFH} 
\startmainpages

\begin{abstract}
    We present a faithful translation from Aloni's logic $\mathsf{BSL}$ into a classical modal logic over semilattices combining ideas from earlier modal translations of truthmaker semantics with algebraic representation theory. We show how our method can also be made to work for other current hyperintensional logics. Finally, we place our results in a broader setting of translating hyperintensional logics into classical modal ones, and discuss applications and new challenges that arise for both areas of philosophical logic connected in this way.
    %
\end{abstract}

\section{Introduction}\label{sec:intro}

\vspace{-0.5ex}

\paragraph{Translations} There are many logical systems in the literature, and the number keeps growing. But sometimes, despite surface differences in presentation and motivation, such systems turn out to be connected by translations. These relate what may have initially looked like quite different or even competing conceptual frameworks. This paper is devoted to establishing a new connection of this kind, as will be explained step by step below.

\vspace{-3ex}

\paragraph{The G\"{o}del translation} For a famous example illustrating a number of key points that are relevant to this paper, consider the case of intuitionistic versus classical logic in the early 20th century. The two systems seemed rivals for capturing what makes mathematics tick, and they polemically disagreed on several prominent classical laws of reasoning. But in 1933 Gödel published his famous translation, let us call it $G$ here, which embeds intuitionistic propositional logic into the classical modal logic $\mathsf{S4}$. This translation uses a modal operator $\Box$ to translate an intuitionistic negation $\neg\varphi$ into $\Box{\sim}$ (now with negation read classically) applied to the translation of $\varphi$, and the same modal strengthening is applied to a classical material implication $\varphi \rightarrow \psi$. Conjunction and disjunction, on the other hand, remain as they are in classical logic. 

While this takes care of the basic intuitionistic logical operators, the translation also contains an independent further step. In line with the intuitionistic core notion of 'having been established' all translated formulas are to be persistent in the sense that $G(\varphi) \rightarrow \Box G(\varphi)$ is valid. This special effect is achieved in part in the translation of the proposition letters $p$, which go to $\Box p$. The modal logic $\mathsf{S4}$ then takes care of lifting this persistence behavior to all translations of intuitionistic propositional formulas. Alternatively, we could have added the persistence as an external special assumption on proposition letters, mapping intuitionistic logic into a `subtheory' of $\mathsf{S4}$. This perspective will be taken in our main result below.

The outcome of all this is that the Gödel translation is faithful in the following sense:

\vspace{1mm}

 For all intuitionistic propositional formulas $\varphi_1, \dots, \varphi_k, \psi$:\vspace{1mm}

	$\varphi_1, \dots, \varphi_k \vDash \psi$ in intuitionistic logic iff  $G(\varphi_1), \dots, G(\varphi_k) \vDash G(\psi)$ in classical modal $\mathsf{S4}$.\footnote{The full result is usually taken to have been proved only in~(\cite{MT48}). However, see~(\cite{HämeenAnttilaPlato2023}) for a different view based on Gödel's unpublished shorthand notes.}
	
\medskip

A few further points are worth noting here as they also apply to the main results in this paper. First, the translation $G$ does not run into just the language of classical propositional logic. This is to be expected as the SAT complexity for classical logic is NP-complete, while  for intuitionistic logic it is PSPACE-complete, a more complex class modulo NP $\neq$ PSPACE. Translations would have to blow up formula size in unnatural ways for this complexity jump. 

But Gödel's added modality $\Box$ is not just a blind added coding gadget since it makes the surplus in the intuitionistic view of negation and implication explicit with a modality that can be read in terms of knowledge or provability instead of truth. This is often the case, interesting translations embody a conceptual idea.\footnote{While this makes $\mathsf{S4}$ look like an extension of intuitionistic logic allowing for non-persistent behavior, a surprising converse faithful translation was given in~(\cite{Fernandez2006}), with an important correction by~\textcite{GoreThomson2019}, suggesting that intuitionistic propositional logic is also a full-fledged theory of modality.}

\vspace{-3ex}

\paragraph{A semantic perspective} One could think of the $G$-translation as a purely syntactic device which relates proofs in intuitionistic logic and classical modal logic. Then major criteria for a good translation would be syntactic or computational simplicity. This is a valid perspective which lends itself to combinatorial proof analysis, and it may well have been that of Gödel himself. Even so, modern presentations of the translation tend to be semantic, and this involves a somewhat richer set-up that we will now explain. 

We are not just syntactically translating formulas from a source logic into those of a target logic, but we are also connecting the models for the two systems. That is, we do not just connect two languages, but also their 'world views'. With $G$, this connection is rather simple: the same relational preorder model can interpret intuitionistic propositional logic and the modal language of $\mathsf{S4}$ (which, as we noted, also allows for non-persistent formulas involving modal diamonds). But note that this setting uses the Kripke semantics for intuitionistic logic introduced in~(\cite{Kripke1965}), whose clauses are  direct renderings of the $G$-translation.\footnote{Inverting the history, one could also say that the Gödel translation just reproduces the meta-language truth conditions for a basic intuitionistic semantics in the object language. Formalizations of truth conditions in fact underlie many translations between logics, cf.~\textcite{Benthem2026}.}

However, in general, translations will involve a pair of a syntactic map $t$ and a simultaneous model transformation $T$ obeying the following general contravariant schema for all formulas $\varphi$ of the source logic:

\medskip

	$\mathsf{M} \vDash t(\varphi)\,\, \text{if and only if} \, \,T(\mathsf{M}) \vDash \varphi.$
	
	\medskip

\noindent This is the perspective that we will need for the main results of this paper.

\vspace{-3ex}

\paragraph{Connecting practices} A translation connects logical systems, but it can also connect broader practices involving these systems. The historical effect of the $G$-translation has been a steady flow of insights both ways, and even a merge of topics between intuitionistic and classical research programs. In this way, the original antagonism has dissipated, and there is a fruitful coexistence of intuitionistic and classical logic in many respects: algebraic, model-theoretic, proof-theoretic, topological, and so on. 

\vspace{-3ex} 

\paragraph{Translations galore} Translations abound in logic, witness the many samples in~(\cite{French2017, BenthemBezhanishviliHolliday2017,Benthem2026}). And new ones keep coming to light, such as from relevant logics or Lambek calculi to classical modal logics (\cite{Knudstorp2026,Kurtunina1995}) and those in~(\cite{HollidayMassas2026}) between much weaker propositional logics than either intuitionistic or classical ones. Interestingly, the $G$-translation returns as a system connector in this realm of weak logics, thus achieving a much broader scope. 

With this background in place, let us now introduce the main topic of this paper.

\paragraph{Hyperintensional logics} The setting for this paper is what is often seen as a new tension today: that between hyperintensional logics versus classical modal logics.\footnote{What hyperintensionality means is a matter that we will not settle here. The original definition by Max Cresswell bestowed this name whenever a system fails to validate some classical equivalences. More interesting forms of hyperintensionality arise when logics deviate from classical ones for some principled reason such as finer modelling of ontological or epistemic distinctions. The systems that we discuss in this paper do just that.}  The systems seem worlds apart, the underlying logics differ: non-classical versus classical, and so on. Hyperintensional logics are usually motivated and developed semantically, and their models and semantics look much richer than those for classical modal logic. But appearances can be deceptive. The additional model structure often comes in the form of more structured worlds and new relations between them. But such relational structures still typically fall within the remit of classical modal logic as well. This is the sense in which \textcite{Benthem2024:kripke} says that, despite some current narratives, doing hyperintensional logic is not doing less, but rather more modal logic. What follows will be in this spirit, though there is more to the issues, and we will also preserve the special nature of each side in this contact.

\vspace{-2ex}

\paragraph{A pilot translation} One motivating example for the analysis in this paper will be a translation of \posscite{Fine17} basic truthmaker logic into a classical modal logic with a binary modality for taking suprema of states. 
We will provide technical details in Section~\ref{sec:TMLtoMIL} below, but here is an explanation of the main ideas. 

Truthmaker logic is interpreted on domains of 'states' ordered by containment, typically on semilattices of these, and it employs two basic semantic notions: '(exact) truthmaking' and '(exact) falsemaking', which are treated on a par in parallel semantic clauses for all logical operators. In particular, a truthmaker negation $\neg\varphi$ switches between these two, thus connecting the two stances. Making a conjunction $\varphi \land \psi$ true is typically done by a fusion or merge of a truthmaker for $\varphi$ and one for $\psi$. For much more on this framework, see (\cite{Fine17,Fine17b}).

While this deviates considerably from classical modal semantics, there are underlying analogies. These same semantic structures also support a classical modal language over preorders (the same semantic realm that had already served so well for intuitionistic logic) with a binary modality $\Sup \varphi \psi$ that is true at those worlds which are suprema of a $\varphi$-world and a $\psi$-world. This system was known from the  literature as 'Modal Information Logic' (\cite{Benthem97}). 

There is in fact a faithful translation of the truthmaker language making this analogy precise, which will be presented in Section~\ref{sec:TMLtoMIL} below as a warm-up for our main result.

\vspace{-2ex}

\paragraph{Richer hyperintensional logics} While this analysis provides a 'proof-of-concept', much more is required to translate current richer hyperintensional logics. Our running example in this paper will be the 'bilateral state-based logic' $\mathsf{BSL}$ of \textcite{Aloni2022} which has been applied extensively by now in the study of natural language semantics and reasoning. We will present the essentials of this system, interpreted semantically over sets of propositional valuations (a form of `team semantics'), and explain some of its subtle behavior in Section~\ref{sec:BSLandMIL}.

In Section~\ref{sec:BSLtoMIL} we then come to our main result: a new faithful translation from $\mathsf{BSL}$ into a modal information logic over semilattices enriched with modalities for suprema over finite sets of elements, either in $\mu$-calculus style or in that of a modal fragment of Weak Second-Order Logic. The translation and the proof of its correctness involve many new ideas, including a central role for a construction reminiscent of Birkhoff's classical representation of finite distributive lattices as set structures. Following our main result for $\mathsf{BSL}$, we will also show how modifications of this translation apply to a wide range of other hyperintensional systems, such as inquisitive logic or propositional team logics.

\vspace{-2ex}

\paragraph{General issues} The proof of our main result is also a case study for more general issues in translating between hyperintensional and classical modal logics, and even between logical systems in general, such as the ones raised in~\textcite{Benthem2026}. A number of such issues will be discussed in Section~\ref{sec:repercussions} below. After that, Section~\ref{sec:summary} will just be a high-level summary.

\vspace{-1ex}

\section{	Setting the scene: Truthmaker logic and modal information logic}\label{sec:TMLtoMIL}

\subsection{Propositional truthmaker logic}
Our starting point are the truthmaker logics introduced by~\textcite{Fine17} (see also~\cite{Fraassen1969} for an early precursor). At their core are new interpretations of the language of propositional logic over binary order structures that allow for binary suprema. There are many options here, from arbitrary preorders with possibly few such suprema to semilattices where all such suprema exist, or perhaps even complete lattices with all, including infinite, suprema. Points or `states' in such structures can be given many interpretations, ranging from parts of the world to information states or even other entities, and accordingly, taking a supremum can be seen as a mereological sum, an information merge, and so on. We will not restrict to any particular interpretation here in order to maintain maximal generality.

Whatever precise structure is chosen, here are the characteristic clauses for interpreting the main propositional operations. The format uses two notions in tandem, making true (denoted by $\vDash^{+}$) and making false (denoted by $\vDash^{-}$) at points $s$ in a model $\mathsf{M}$ whose valuation {$V=(V^+, V^-)$ indicates where proposition letters $p$ are made true (value $V^+(p, s) = 1$) and false (value $V^-(p, s) = 1$). However, unlike in classical logics, there is no presumption that these two values are exhaustive or exclusive, proposition letters might be made neither false nor true, or both true and false:
\medskip

$\mathsf{M}, s \vDash^{+} p \quad \text{iff} \quad V^+(p, s) = 1$ 
\vspace{1mm}

$\mathsf{M}, s \vDash^{-} p \quad \text{iff} \quad V^-(p, s) = 1$ 
\vspace{1mm}

$\mathsf{M}, s \vDash^{+}  \varphi \land \psi \quad \text{iff} \quad \text{there are} \,\,t, u \,\,\text{with} \,\, (i)\,\, s = \sup(t, u), (ii)\,\, \mathsf{M}, t \vDash^{+}  \varphi, (iii)\,\, \mathsf{M}, u \vDash^{+} \psi$
\vspace{1mm}

$\mathsf{M}, s \vDash^{-}  \varphi \land \psi \quad \text{iff} \quad \mathsf{M}, s \vDash^{-}  \varphi \,\, \text{or}\,\, \mathsf{M}, s \vDash^{-} \psi$
\vspace{1mm}

$\mathsf{M}, s \vDash^{+}  \varphi \lor \psi \quad \text{iff} \quad  \mathsf{M}, s \vDash^{+}  \varphi \,\, \text{or}\,\, \mathsf{M}, s \vDash^{-} \psi$
\vspace{1mm}

$\mathsf{M}, s \vDash^{-}  \varphi \lor \psi \quad \text{iff} \quad \text{there are} \,\,t, u \,\,\text{with} \,\, (i)\,\, s = \sup(t, u), (ii)\,\, \mathsf{M}, t \vDash^{-}  \varphi, (iii)\,\, \mathsf{M}, u \vDash^{-} \psi$

\vspace{1mm}

$\mathsf{M}, s \vDash^{+}  \neg\varphi \quad \text{iff} \quad \mathsf{M}, s \vDash^{-}  \varphi $

\vspace{1mm}

$\mathsf{M}, s \vDash^{-}  \neg\varphi \quad \text{iff} \quad \mathsf{M}, s \vDash^{+}  \varphi $

\medskip

This  makes conjunction and disjunction rich modalized non-classical notions that involve the structural operation of merging states.\footnote{ This view is more akin to the `non-Boolean' conjunctions and disjunctions that can be found in natural language where say ``John and Mary quarrel'' refers to a collective object constructed out of the two objects `John' and `Mary'.} By contrast, negation is not modalized (as it was in intuitionistic logic), but just switches truthmaking and falsemaking, leading to several analogies with classical negation (e.g., Double Negation and the De Morgan laws hold, for more see~\cite{Korbmacher22, FineJ19:rsl}), but in other ways, with a very different intent.

In this setting, truthmaker semantics contains a lot of interesting further notions that we will not discuss in depth here. For instance, the notion of truthmaking is `exact', which entails that a state can make a formula true while none of its proper superstates (or substates) in the ordering do, and which causes, e.g., simplification $p\land q\vdash p$ to fail. And also, the `bilateral' semantic format with decoupled truthmaking and falsemaking allows for a plethora of different notions of consequence between formulas. See again~(\cite{Fine17, Fine16}) for a development of the framework and its philosophical applications.

\vspace{-1ex}

\subsection{Modal information logic}

The idea of using suprema in an ordering as natural tools for semantic interpretation was proposed in classical modal logic in~(\cite{Benthem97}) as an extension of the classical system $\mathsf{S4}$ meant to provide a basic modal theory of the pervasive notion of information. If we think of the objects in the usual preorder models $\mathbb{M} = (W, \leq, V)$ for the $\mathsf{S4}$ language as information states, it makes sense to talk about merging such states, and describing such merges with a binary modality $\Sup\varphi\psi$. The key semantic clause for this `Modal Information Logic' $\mathsf{MIL}$ runs as follows:

\medskip


$\mathbb{M}, x \Vdash  \Sup\varphi\psi  \quad \text{iff} \quad \text{there are} \,\,y, z \,\,\text{with} \,\, (i)\,\, x = \sup(y, z), (ii)\,\, \mathbb{M}, y \Vdash  \varphi, (iii)\,\, \mathbb{M}, z \Vdash \psi$

\medskip

\noindent The clauses for $\bot,\top, {\sim}, \vee,\land$ in this system are the usual Boolean ones. 

There is also a natural two-sided enrichment of the language with one more
binary modality $\Inf\varphi\psi$ dually describing infima in the ordering.
One can define the existential unary backward-looking modality $\Past\varphi$ as follows: $\Sup\varphi\top$, which can be seen by verifying that 

\vspace{-3.5ex}
\begin{align*}
        &x\Vdash \Sup\varphi\top && \text{iff} && \text{there exists $y\leq x$ such that } y\Vdash \varphi
    \end{align*}
    
    \vspace{-1ex}

We can then also define the universal downward modality $\PastH\varphi$, and mutatis mutandis, we can define upward modalities by using $\Inf\varphi\psi$.\footnote{We use the standard tense-logical notation for past-looking modalities here, though of course, our models do not have a primary temporal interpretation. Tense logics of lattice structures are studied in~(\cite{WangWang2025}).}

\vspace{1ex}

$\mathsf{MIL}$ is a classical modal logic forming a natural extension of $\mathsf{S4}$ with models interpreted as families of information states, an interpretation still close to the intuitionistic logic that we started with as a point of entry into our topic. But the language with binary modalities can also be interpreted on many other relational models, in the minimal case: arbitrary accessibility relations as used for the minimal modal logic. To make that work, instead of $x = \sup(y, z)$, the binary modality $\Sup\varphi\psi$ just needs an arbitrary ternary relation $Rx, yz$ on models which can be interpreted in many ways: information merge, composition of actions, or even literal  syntactic juxtaposition of linguistic expressions.

Modal Information Logic still shares many features with classical modal logics. It is a normal modal logic in the technical sense, and its expressive power can be analyzed using an extended notion of bisimulation, its formulas can be translated into a first-order language over its relational models, leading to axiomatizability in principle, and so on. Some further features will emerge in the discussions to follow here.

\vspace{-2ex}

\subsection{A translation}\label{subsec:TMLtoMIL}

It should be clear at a glance that truthmaker semantics and $\mathsf{MIL}$ have a lot in common. This is made precise by the following translation first proposed in~(\cite{Benthem19:jpl}):

\vspace{-2.5ex}

\begin{alignat*}{2}
    (p)^+                 &\quad = \quad p_+                      \qquad\qquad & (p)^-                 &\quad = \quad p_-\\
    (\neg \varphi)^+      &\quad = \quad \varphi^-               \qquad\qquad & (\neg \varphi)^-      &\quad = \quad \varphi^+\\
    (\varphi\wedge \psi)^+&\quad = \quad \Sup \varphi^+\psi^+    \qquad\qquad & (\varphi\wedge \psi)^-&\quad = \quad \varphi^-\lor \psi^-\\
    (\varphi\vee \psi)^+  &\quad = \quad \varphi^+\lor \psi^+    \qquad\qquad & (\varphi\vee \psi)^-  &\quad = \quad \Sup \varphi^-\psi^-
\end{alignat*}

\vspace{1mm}

Here is an intuitive explanation of how this works. Two main ideas are involved.

First, truthmaker semantics is a four-valued system.
While this seems vastly different from classical two-valued logics, the four-valuedness can be taken care of quite simply by defining \emph{two} entangled translations, one for the `true' case and one for the `false' case. At the atomic level, this just needs a duplication of proposition letters $p$ into two variants $p_{+}$ and $p_{-}$ with obvious intuitive meanings.\footnote{In fact, a similar parallel-translation format will work for classical translations of any many-valued logic with a fixed finite number of truth values.} 

Note how working with two translations allows us to mimic any semantics which treats truthmaking and falsemaking for a logical operation as independent design options. And the switching role of negation is reflected in switching between the two models of the translation.

There is one feature here that may need some explanation. In this new classical language, there can be four truth value combinations for the two proposition letters $p_{+}$,  $p_{-}$. One can be true, the other false, neither can be true, or there can even be points where both are true.
If we want to suppress the latter option, the translation will have to work on models satisfying a classical modal formula like $\Box{\sim}(p_{+} \land p_{-})$, where $\Box$ is our earlier downward universal modality. 


\vspace{1ex}

Here is the effect of the translation when using `truthmaking-to-truthmaking' consequence $\vDash ^{+}$ between formulas in the propositional truthmaker language, and ordinary classical consequence $\vDash$ for the modal information language:

\begin{theorem} For all truthmaker formulas $\varphi_1, \dots, \varphi_k, \psi$, the following are equivalent:

\vspace{1ex}

(i) $\varphi_1, \dots, \varphi_k \vDash ^{+} \psi$,

\vspace{1ex}

(ii) $(\varphi_1)^{+}, \dots, (\varphi_k)^{+} \vDash (\psi)^{+}$
\end{theorem}

Many variants of truthmaker consequence can be translated faithfully in a similar way (cf.~\cite{Knudstorp2023:synthese}). This translation works on any class of models for both systems: preorders, semilattices, but also arbitrary ternary relations (if one were to generalize the semantics of truthmaker logic that far). 

\vspace{-2ex}

\subsection{Comparing the two sides}\label{subsec:comparison}

Since the preceding translation appeared, a number of results have illuminated the above connection. For instance, motivated by the connection to modal logic,~\textcite{Knudstorp2023:synthese} proves a truthmaker analog of the finite model property and develops a `standard translation' into first-order logic ($\mathsf{FOL}$) to transfer compactness and recursive enumerability. And in fact, this translation into $\mathsf{FOL}$ may simply be viewed as the composition of the above translation with the standard translation of $\mathsf{MIL}$ into $\mathsf{FOL}$. Other connections between the two sides involve the transfer of basic classical modal notions such as bisimulation to the design of similar notions for truthmaker logics, witness the `simulation pairs' of \textcite{Konstantinou2025}.
   
The contact between a hyperintensional logic and a classical modal logic established in this way also raises many issues for both realms. One very basic question concerns the computational complexity of valid consequence. The basic truthmaker logic over semilattices or associative preorders is decidable, as shown in~(\cite{FineJ19:rsl, Knudstorp2023:synthese}).\footnote{Associativity of a ternary ordering $R$ means that (i) if $Rsxw$ and $Rxuv$, then there exists an $y$ such that $Rsuy$ and $Ryvw$, (ii) the same `regrouping' in the direction from right to left.} But the basic modal logic with a binary modality over associative orders or semilattices is undecidable~(\cite{Kurucz93,KuruczNSS95:jolli,Knudstorp2026b}).\footnote{Although it becomes decidable when the modality refers to suprema, or minimal upper bounds, taken over pre- or partial orders (\cite{Franssen2025,Knudstorp2023:jpl,Knudstorp2023:tbillc, FranssenKnudstorp2026b}).} 

What is the explanation of this mismatch? One important feature here is that the above translation takes arbitrary formulas of truthmaker logic into syntactically \emph{positive} formulas of $\mathsf{MIL}$ which do not use classical negation. Positive fragments of logical systems are natural and they play an important role in computational logic or database theory. This feature of the translation illuminates what underlies the decidability of truthmaker logic, though it would go too far to say that hyperintensional logics are positive fragments of classical modal ones. Indeed, the role of classical negation or substitutes for it in hyperintensional logics can be subtle, witness the analyses in~(\cite{Knudstorp2026b,GalatosKnudstorpetal2026}). But the point remains: natural basic questions arise in a translational perspective.

Our new translation in this paper will extend the above notions, results and issues  to a comparison between two richer  logical systems.

\section{Bilateral state-based logic and enriched modal information logic}\label{sec:BSLandMIL}	Here are the two logical systems at issue in the new translation result presented in this paper.

\vspace{-2ex}

\subsection{$\mathsf{BSL}$ over propositional valuation models}


\noindent   $\mathsf{BSL}$ is a variant of $\mathsf{BSML}$, a bilateral team-based modal logic developed in (\cite{Aloni2022})  to explain   phenomena of free choice (see also \cite{Anttila2024}) as well as other puzzling phenomena in natural language semantics. In free choice inferences, a disjunction in the scope of a possibility modal gives rise to conjunctive meanings  (\textit{You may go to the beach or to the cinema $\leadsto$ You may go to the beach}, \cite{Kamp1973}).  $\mathsf{BSL}$, which defines a   non-persistent $\blacklozenge$-modality    (a relative of the {\sc might} from   \cite{Veltman1996})  focuses on epistemic free choice inferences (\textit{Mr.~X might be in Victoria  or    in Brixton 
         $\leadsto$ Mr.~X might be in
      Victoria}, \cite{Zimmermann2000}).

 The  core idea behind $\mathsf{BS(M)L}$ is that free choice and related inferences originate from a tendency in human cognition to  avoid emptiness, something  modeled   by the non-emptiness atom \NE \ from team logics (\cite{YangVäänänen2017}) which interpret formulas with respect to sets of possibilities (teams). \NE \ is supported in a team iff the team is not empty. The main result of $\mathsf{BS(M)L}$ is that by disallowing empty teams using \NE \ we can derive free choice inferences 
 
 \vspace{-1ex}
 
 $$ \blacklozenge  ((\alpha \wedge \NE) \vee (\beta \wedge \NE))  \vDash  \blacklozenge  \alpha$$

\noindent   as well as  their   cancellations under negation: $\neg  \blacklozenge ((\alpha \wedge \NE) \vee (\beta \wedge \NE) ) \vDash \neg \blacklozenge  \alpha$. 

The latter result relies on the adopted bilateralism, where, like in truthmaker semantics, each connective comes with a support  and an antisupport condition with negation as a switch.
Also crucial to $\mathsf{BS(M)L}$ is  the adoption of split disjunction. A team $t$ supports a split disjunction  iff  $t$ is the union of two subteams each supporting one of the disjuncts.  In a sense this gives disjunction a conjunctive flavor as `combination' of alternatives. Our later translation of   split disjunction will use the    $\Sup$-modality making explicit its modalized non-classical nature.
\medskip 

We now proceed to technical details. We fix a non-empty finite set of proposition letters $\mathsf{P}$. The $\mathsf{BSL}$-language over $\mathsf{P}$, denoted $\mathsf{Fm}(\mathsf{P})$, is generated by the following grammar, where $p\in \mathsf{P}$:

\vspace{-1ex}
\[
    \varphi\mathrel{::=} \bot\hspace{0.1cm}|\hspace{0.1cm} \NE\hspace{0.1cm}|\hspace{0.1cm}p \hspace{0.1cm}| \hspace{0.1cm} \neg\varphi\hspace{0.1cm}|\hspace{0.1cm}\varphi\lor\varphi\hspace{0.1cm}|\hspace{0.1cm}\varphi\land\varphi\hspace{0.1cm}|\hspace{0.1cm}\blacklozenge \varphi
\]
\vspace{-2ex}

A \textit{team} $t\subseteq \{v\mid v:\mathsf{P}\to\{0,1\}\}$ is a set of classical valuations over $\mathsf{P}$. A team $t$ \textit{supports} (resp.\ \textit{antisupports}) a formula $\varphi\in \mathsf{Fm}(\mathsf{P})$, written $t\vDash\varphi$ (resp.\ $t\Dashv\varphi$), according to these clauses:

\vspace{-3.5ex}

\begin{align*}
        &t\vDash p && \text{iff} && \text{for all } v\in t, v(p)=1\\
        &t\Dashv p && \text{iff} && \text{for all } v\in t, v(p)=0\\
        &t\vDash \bot && \text{iff} && t=\varnothing\\
        &t\Dashv \bot && \phantom{\text{iff}} && \text{always}\\
        &t\vDash \NE && \text{iff} && t\neq \varnothing\\
        &t\Dashv \NE && \text{iff} && t= \varnothing\\
        &t\vDash \neg\varphi && \text{iff} && t\Dashv \varphi\\
        &t\Dashv \neg\varphi && \text{iff} && t\vDash \varphi\\
        &t\vDash \varphi\land\psi && \text{iff} && t\vDash \varphi \text{ and } t\vDash \psi\\
        &t\Dashv \varphi\land\psi && \text{iff} && \text{there are } t',t'' \,\,\text{with} \,\, (i)\,\, t = t'\cup t'', (ii)\,\, t' \Dashv  \varphi, (iii)\,\, t''\Dashv \psi\\
        &t\vDash \varphi\lor\psi && \text{iff} &&\text{there are } t',t'' \,\,\text{with} \,\, (i)\,\, t = t'\cup t'', (ii)\,\, t' \vDash  \varphi, (iii)\,\, t''\vDash \psi\\
        &t\Dashv \varphi\lor\psi && \text{iff} && t\Dashv \varphi \text{ and } t\Dashv \psi\\
        &t\vDash \blacklozenge\varphi && \text{iff} && \text{there is a non-empty } t'\subseteq t \text{ such that } t'\vDash \varphi\\
        &t\Dashv \blacklozenge\varphi && \text{iff} && t\Dashv \varphi
\end{align*}

We have chosen $\mathsf{BSL}$ as a main focus in this paper because it is a full-fledged hyperintensional logic with serious applications in the semantics of natural language. But this also means that it has some special features which make extending the earlier translation for truthmaker logic into $\mathsf{MIL}$ a non-trivial task. Here are two, both having to do with special closure properties that denotations of formulas must have in the system.

First, as in team semantics,  the direction of growth of information for our teams is in the downward direction. But while indeed the larger part of the above syntax generates formulas which are downward persistent under support, there are notable exceptions. In particular, the crucial non-emptiness atom $\NE$ which underlies the existential modality $\blacklozenge$ is not persistent, since it will fail in the empty team.\footnote{This situation is somewhat analogous to that in modal $\mathsf{S4}$ where existential modalities may become false as we go upward in the information ordering. Incidentally, on a general relational perspective of preorders, which direction we consider as information growth is just a conventional choice.}

Next, unlike in many team semantics, there are also closure conditions in the opposite direction. In particular, for reasons of fit with natural language, $\mathsf{BSL}$ endorses the validity of 

\vspace{-1.5ex}

$$(\varphi \lor \varphi) \vDash \varphi$$

\noindent which can be seen as a sort of `sure thing principle': if a current information state can actually be decomposed into two richer components both supporting $\varphi$, then we should endorse $\varphi$ right now. The technical effect of this requirement is that denotations of formulas will be closed under taking binary suprema, and taking this iteratively, under taking arbitrary suprema of non-empty finite subsets. In other words, we are now talking about denotations that behave more or less like ideals of sets.\footnote{The denotation $\lVert\alpha\rVert$ of formulas $\alpha$ free of $\NE$ and $\blacklozenge$ are ideals of sets, but $\mathsf{BSL}$-denotations need, more generally, only be `weak ideals': closed under binary suprema and under \textit{convexity}, i.e. if $t,t''\in \lVert\varphi\rVert$ and $t\subseteq t'\subseteq t''$, then $t'\in \lVert\varphi\rVert$. Weak ideals, in fact, characterize the denotations of $\mathsf{BSL}$: every $\mathsf{BSL}$-denotation is a weak ideal of sets, and conversely, every weak ideal of sets is the denotation of a $\mathsf{BSL}$-formula~(\cite{AnttilaKnudstorp}). Interestingly, weak ideals are also central in truthmaker semantics, where they make up the regular propositions or the replete content of a formula~(\cite{Fine16,Fine17b, Jago2017}).} This feature, too, must find its place in our translation.


\subsection{Enriched modal information logics over semilattice models}
The modal target language to be used in translating $\mathsf{BSL}$ has a syntax extending that of $\mathsf{MIL}$ introduced above. As explained above, we need to be able to talk in the language about the result of finitely iterating binary suprema. This can be done in two natural ways.
%
%

\vspace{-2ex}

\paragraph{Modal information logic with finite suprema} What we will  need for the purposes of our translation is closure under suprema of finite sets of points in our models. One modal logic  serving this purpose has the
following syntax:

\vspace{-1.5mm}
\[
    \varphi\mathrel{::=} \bot\hspace{0.1cm}|\hspace{0.1cm} p \hspace{0.1cm}| \hspace{0.1cm} {\sim}\varphi\hspace{0.1cm}|\hspace{0.1cm}\varphi\lor\varphi\hspace{0.1cm}|\hspace{0.1cm}\Sup\varphi\varphi\hspace{0.1cm}|\hspace{0.1cm}\Sast\varphi
\]

\vspace{1mm}
Here the modality $\Sast\varphi$ is true at the points which are suprema of some non-empty finite set of points satisfying $\varphi$.  `$\top$' and `$\land$' can, as before, be defined in the usual way. 

We can still view this as a modal logic, as, by a simple argument, this language is invariant for a notion of modal bisimulation with back-and-forth clauses for suprema on either side. 

But as for the validities, note that one may have $x\Vdash \Sast (p\lor q)$ yet $x\nVdash \Sast p\lor \Sast q$, so $\Sast$ does not distribute over disjunctions, nor does it distribute over conjunctions. Still we do have $\Sast \bot\equiv \bot$, while, more generally, this modality is still upward monotone allowing for weakening its propositional argument.\footnote{A slightly different clause for $\Sast$ would allow the finite set of points in the above definition of $\Sast$ to be empty. For more on this variant, see~(\cite{FranssenKnudstorp2026}) which includes a translation of the equational theory of lattices into a modal information logic with this modality (see also~\cite{Franssen2025}).}

This system is interesting in its own right. It can be viewed as a modal fragment of \emph{Weak Second-Order Logic} $\mathsf{WSOL}$ where quantification over sets is restricted to finite sets only. While very natural for many purposes, not all that much is known about $\mathsf{WSOL}$ (see \textcite{LopezEscobar1967} for some early results), let alone its modal fragments.

\vspace{-2ex}

\paragraph{Our work space: distributive semilattices} In what follows, we will work with distributive semilattices as a close abstract approximation to the valuation semantics for $\mathsf{BSL}$ given above. 

\medskip

A \textit{semilattice} is a partial order with all binary suprema, and it is \textit{distributive} if for all $x,y,z$:

\vspace{-3mm}

\[
    x\leq y\curlyvee z \text{ only if there are $y'\leq y$ and $z'\leq z$ such that $x=y'\curlyvee z'$.}
\]

\vspace{-0.5mm}
Here, `$\leq$' is the partial order induced by the semilattice operation, i.e. $x\leq y$ iff $x\curlyvee y =y$.

\vspace{1ex}

Modal information languages can be interpreted in an obvious way in \emph{pointed semilattice models} ($\mathbb{M}, x$) consisting of a semilattice with a valuation sending proposition letters to subsets of the semilattice plus a distinguished point $x$. The general modal clauses that we have introduced earlier apply, with the understanding that the earlier (in general, partial) binary supremum operation $\sup(x, y)$ now becomes the lattice operation $x\curlyvee y$.

\vspace{1ex}

In this special setting, another modal logic also becomes a serious translation target.

\vspace{-2ex}

\paragraph{Modal information logic with fixed-points}  The system $\mu\mathsf{MIL}$  extends the earlier syntax of modal information logic with a smallest fixed-point operator. This extension needs some explanation. The standard modal $\mu$-calculus has only unary modalities, but this time we also allow the binary modality $\Sup \varphi\psi$ of $\mathsf{MIL}$. Apart from that, positive occurrence and fixed-points are defined precisely as for the $\mu$-calculus. 

In particular, this language allows us to define the following fixed-point formula: 

\vspace{-1.5ex}

$$\mu p. (\varphi \lor \Sup pp)$$


\noindent This defines the set of all points that can be decomposed by finite iterations into a set of points satisfying $\varphi$. On semilattices (though not on arbitrary preorders), this is exactly the same set as that defined by the earlier modality $\Sast\varphi$ in modal $\mathsf{WSOL}$.

This fragment of $\mu\mathsf{MIL}$ uses only a small part of the full system, viz. a fixed-point that is always reached in $\omega$ steps in any model where it is computed.  The resulting modal fixed-point logic  shows some analogies with Propositional Dynamic Logic $\mathsf{PDL}$, and in fact, with another interpretation in mind (composition of actions), \textcite{Benthem1994} shows how the standard decidability and completeness proof for $\mathsf{PDL}$ transfers to this setting with a binary modality.

\medskip
The equivalence between the two modal definitions for finite suprema breaks down when we generalize the semantics to pre- or partial orders, in which case the modal $\mathsf{WSOL}$ and $\mu\mathsf{MIL}$ become genuine alternative systems for introducing infinitary notions into a modal language. Both are  of interest in their own right as classical modal logics.

\medskip

Finally, in setting up our translation of $\mathsf{BSL}$, like with the translation of truthmaker consequence (\ref{subsec:TMLtoMIL}), we use a `split version' of the above syntax with $p\in \mathsf{P}_\pm\mathrel{:=}\{p_+, p_-\mid p\in\mathsf{P}\}$. That is, for each proposition letter $p\in \mathsf{P}$ from the team-semantical language, we consider two proposition letters $p_+$ and $p_-$ in the modal language.

\vspace{-1ex}

\section{The translation}\label{sec:BSLtoMIL}

\subsection{The syntactic translation}
    
Our goal is to define a faithful translation 
$\cdot^+$ sending formulas of $\mathsf{BSL}$ to formulas of the above extended modal information language with the following property.  For all $\varphi, \psi\in \mathsf{Fm}(\mathsf{P})$:

\vspace{-3ex}

\begin{align}
    \varphi\vDash\psi\qquad \text{iff} \qquad \Gamma, \varphi^+\Vdash \psi^+\label{iff}
\end{align}


\begin{convention}
    In the above and henceforth, `$\vDash$' refers to semantic support and consequence within $\mathsf{BSL}$, while `$\Vdash$' denotes semantic satisfation and consequence within the modal logic.
\end{convention}

We state the definitions of $\cdot^+$ and $\Gamma$ now, deferring  explanation to the subsection after this.  For a start, we set
\[
    \Gamma\mathrel{:=}\left\{\PastH(\NE^+\lor \NE^-),
    \bigwedge_{p\in \mathsf{P}}\Sup p^+p^-\right\}, 
\]
and then define $\cdot^+$ together with another, interweaved translation $\cdot^-$ (for antisupport $\Dashv$):
\begin{alignat*}{2}
    p^+                      &\quad \mathrel{:=} \quad \PastH\Sast p_+ 
    \qquad\qquad\qquad& p^-                      &\quad \mathrel{:=} \quad \PastH\Sast p_- \\
    \NE^+                    &\quad \mathrel{:=} \quad \bigwedge_{p\in \mathsf{P}}{\sim}(p^+\land p^-) 
    \qquad\qquad\qquad& \NE^-                    &\quad \mathrel{:=} \quad \bigwedge_{p\in \mathsf{P}}(p^+\land p^-) \\
    \bot^+                   &\quad \mathrel{:=} \quad \NE^- 
    \qquad\qquad\qquad& \bot^-                   &\quad \mathrel{:=} \quad \top \\
    (\neg \varphi)^+         &\quad \mathrel{:=} \quad \varphi^- 
    \qquad\qquad\qquad& (\neg \varphi)^-         &\quad \mathrel{:=} \quad \varphi^+ \\
    (\varphi\vee \psi)^+     &\quad \mathrel{:=} \quad \Sup\varphi^+\psi^+ 
    \qquad\qquad\qquad& (\varphi\vee \psi)^-     &\quad \mathrel{:=} \quad \varphi^-\land \psi^- \\
    (\varphi\land \psi)^+    &\quad \mathrel{:=} \quad \varphi^+\land \psi^+ 
    \qquad\qquad\qquad& (\varphi\land \psi)^-    &\quad \mathrel{:=} \quad \Sup\varphi^- \psi^- \\
    (\blacklozenge\varphi)^+ &\quad \mathrel{:=} \quad \Past(\NE^+\land \varphi^+) 
    \qquad\qquad\qquad& (\blacklozenge\varphi)^- &\quad \mathrel{:=} \quad \varphi^-
\end{alignat*}

As we have seen in Section~\ref{sec:intro}, a translation that aims to connect meanings also has a semantic aspect, in the form of a transformation between the models for its logics. We will provide such transformations for $\mathsf{BSL}$ and our modal information logics. However, before defining these, we give an extensive explanation of the clauses of the above syntactic definition. To do so, we also repeat some technical points in our warm-up example of truthmaker logic and $\mathsf{MIL}$.

\vspace{-2ex}

\subsection{Explaining the details of the syntactic translation}

\paragraph{Valuations and ideals} For proposition letters $p\in \mathsf{P}$, we have set
\begin{align*}
    p^+                      &\quad \mathrel{=} \quad \PastH\Sast p_+ 
    \qquad& p^-                      &\quad \mathrel{=} \quad \PastH\Sast p_-
\end{align*}
Two features stand out: the use of a pair of translations, $\cdot^+$ and $\cdot^-$, rather than a single translation, and the prefixing with $\PastH\Sast$. Each corresponds to a semantic feature of $\mathsf{BSL}$.

Indeed, $\cdot^-$ translates the full $\mathsf{BSL}$-language, on a par with $\cdot^+$. Just as in our warm-up example, this is due to meaning in $\mathsf{BSL}$ having both support ($\vDash$) and antisupport ($\Dashv$) as  primitive, requiring a pair of translations to achieve this bilateralism inside the unilateral modal logic. This also explains the doubling of proposition letters from $\mathsf{BSL}$'s $\mathsf{P}$ to the modal logic's $\mathsf{P}_\pm$.\footnote{Please do not confuse the subscripted $p_+\in\mathsf{P}_\pm$, a proposition letter in the modal language, and the superscripted $p^+$, the translation of the proposition letter $p$ in the $\mathsf{BSL}$ language, i.e. $p^+=\PastH\Sast p_+\neq p_+$. Ditto for $p_-$, $p^-$.}

The prefix `$\PastH\Sast$' is best thought of in analogy to the Gödel translation mapping $p$ to $\Box p$. In more detail, note that team-semantical models form powersets, and that the denotations of proposition letters $p$ form \emph{ideals} on these powersets:
\begin{itemize}
    \item (downward closure) if $t\vDash p$ and $t'\subseteq t$, then $t'\vDash p$ \hfill [likewise for $\Dashv$]
    \item (closure under binary joins) if $t\vDash p$ and $t'\vDash p$, then $t\cup t'\vDash p$ \hfill [likewise for $\Dashv$]
\end{itemize}
These characteristics are structurally important (unlike the fact that they are \emph{principal} ideals). In particular, they ensure distributivity for proposition letters and idempotence of $\lor$:
\begin{itemize}
    \item (by downward closure) $p\land (q\lor r)\vDash (p\land q)\lor (p\land r)$
    \item (by closure under binary joins) $p\lor p\vDash p$
\end{itemize}
The prefix `$\PastH\Sast$' mirrors these restrictions. The $\PastH$-prefix ensures that the denotation of $p^+$ (or $p^-$) is downward closed, and the $\Sast$-prefix that it is closed under binary joins.

\paragraph{Valuations and gaps and gluts} While $\mathsf{BSL}$-meaning is bilateral, the support- and antisupport conditions for proposition letters are not independent of each other, unlike the case of truthmaker semantics. They are instead subject to certain interplay leaving neither `gaps' (for all teams $t$, $t\vDash p\lor \neg p$) nor `gluts' (for all teams $t$, $t\vDash p\land \neg p$ iff $t=\varnothing$). The translation models this through a set $\Gamma $ of auxiliary assumptions:
\[
    \Gamma\mathrel{=}\left\{\PastH(\NE^+\lor \NE^-),
    \bigwedge_{p\in \mathsf{P}}\Sup p^+p^-\right\}
\]
The latter, $\bigwedge_{p\in \mathsf{P}}\Sup p^+p^-$, is translated excluded middle, and its inclusion in $\Gamma$ is comparable to the folklore translation of classical propositional logic $\mathsf{CPL}$ into strong Kleene logic $\mathsf{K3}$:
\[
    \varphi\vDash_{\mathsf{CPL}}\psi\qquad \text{iff}\qquad \bigwedge_{p\in\mathsf{P}}(p\lor \neg p),\varphi\vDash_{\mathsf{K3}}\psi.\footnote{If one takes teams to be sets of $\mathsf{K3}$-valuations, i.e. \emph{partial} functions $v:\mathsf{P}\to\{0,1\}$, rather than $\mathsf{CPL}$-valuations, the resulting logic translates under $\cdot^+$ by omitting $\bigwedge_{p\in \mathsf{P}}\Sup p^+p^-$ from $\Gamma$, and instead adding $\PastH\Past\NE^-$ (as $\bigwedge_{p\in \mathsf{P}}\Sup p^+p^-$ does double duty, also entailing $\PastH\Past\NE^-$). We will return to this in more detail later.}
\]

For the former part of $\Gamma$, $\PastH(\NE^+\lor \NE^-)$, a similar story applies, as the translation clauses for $\NE$ make clear:
\begin{alignat*}{2}
    \NE^+                    &\quad \mathrel{=} \quad \bigwedge_{p\in \mathsf{P}}{\sim}(p^+\land p^-) 
    \qquad\qquad\qquad& \NE^-                    &\quad \mathrel{=} \quad \bigwedge_{p\in \mathsf{P}}(p^+\land p^-)
\end{alignat*} 

\vspace{-3ex}

\paragraph{Connectives} Last, the connectives (and constant $\bot$) were translated in the expected way:
\begin{alignat*}{2}
    \bot^+                   &\quad \mathrel{=} \quad \NE^- 
    \qquad\qquad\qquad& \bot^-                   &\quad \mathrel{=} \quad \top \\
    (\neg \varphi)^+         &\quad \mathrel{=} \quad \varphi^- 
    \qquad\qquad\qquad& (\neg \varphi)^-         &\quad \mathrel{=} \quad \varphi^+ \\
    (\varphi\vee \psi)^+     &\quad \mathrel{=} \quad \Sup\varphi^+\psi^+ 
    \qquad\qquad\qquad& (\varphi\vee \psi)^-     &\quad \mathrel{=} \quad \varphi^-\land \psi^- \\
    (\varphi\land \psi)^+    &\quad \mathrel{=} \quad \varphi^+\land \psi^+ 
    \qquad\qquad\qquad& (\varphi\land \psi)^-    &\quad \mathrel{=} \quad \Sup\varphi^- \psi^- \\
    (\blacklozenge\varphi)^+ &\quad \mathrel{=} \quad \Past(\NE^+\land \varphi^+) 
    \qquad\qquad\qquad& (\blacklozenge\varphi)^- &\quad \mathrel{=} \quad \varphi^-
\end{alignat*}

Except for the last clause, this was how the earlier translation of truthmaker logics worked, be it that now the roles of disjunction and conjunction have been flipped. The final clause ensures that the modality $\blacklozenge\varphi$ of $\mathsf{BSL}$ gets its intended meaning.


\subsection {Comparing the two translations}

Beyond the preceding technical perspectives, it may help to compare this new translation of bilateral state logic $\mathsf{BSL}$ with the earlier one for truthmaker logic $\mathsf{TS}$ in terms of the different systems they need to connect. Their hyperintensionality shows at different places. For instance, the following principle fails in $\mathsf{TS}$:

\vspace{-1.5ex}

 $$\varphi \wedge \neg \varphi \equiv \psi \wedge \neg \psi$$

\noindent But in $\mathsf{BSL}$, it is rather the following equivalence that fails (which also fails in $\mathsf{TS}$): 

\vspace{-2mm}

$$\varphi \vee \neg \varphi \equiv \psi \vee \neg \psi$$

Now our translations using $\Sup$ do explain the hyperintensionality at stake here as stemming from a common  cause, since $\mathsf{TS}$ conjunction and $\mathsf{BSL}$ disjunction both become binary modalities, and in general we do not have the modal validity $\Sup\varphi^+ \varphi^- \equiv \Sup\psi^+ \psi^-$.

However, conceptually, the motivations for such failures are quite different in the two logics. 

In $\mathsf{TS}$,  at least in the original motivation, formulas describe worldly structures and   $\Sup$  models the merging or fusion of   different portions of the world: the  fusion of a truthmaker for $p$ and a truthmaker for $\neg p$ can differ   from the fusion of a truthmaker for $q$ and a truthmaker for $\neg q$.  Thus, $p \wedge \neg p \not\equiv q\wedge \neg q$. $\mathsf{BSL}$, in contrast,  is motivated by  natural language phenomena, and $\Sup$ here models the merging of  epistemic possibilities. When saying \textit{A or B} one normally conveys  that  it might be $A$ and it might be  $B$. Thus disjunction  can be viewed as a merge of epistemic options.
Considering  $\varphi$ and $\neg \varphi$ as both epistemically possible is different from considering $\psi$ and $\neg \psi$ as such. Thus
$\varphi \vee \neg \varphi \not\equiv \psi \vee \neg \psi$. For a concrete counter-example, take $\varphi = (p \wedge \NE)$ and $\psi = (q \wedge \NE)$.\footnote{Other counter-examples involve  free choice sentences. When saying \textit{It might be  A or B} you normally convey  that  it might be $A$ and it might be $B$. On the other hand, when saying {\it It might not be A or B} one normally conveys it might not be $A$ and it might not be $B$. Thus   \textit{It might be  A or B} and its negation do not cover the whole logical space: a speaker informed that $A$ is the case and not $B$, would reject both  sentences. Thus  we have another  violation of the Law of Excluded Middle, and, as a consequence, another violation of the above equivalence, this time with $\varphi = \blacklozenge ((p \wedge \NE) \vee (q \wedge \NE))$ and $\psi = \blacklozenge ((r \wedge \NE) \vee (s \wedge \NE))$.}

The different motivation for the two logics, description of ontological structures in
the world ($\mathsf{TS}$) versus explaining patterns in inferential and
communicative human activities ($\mathsf{BSL}$),   is reflected in our translations of the atoms combined with additional restrictions on the target variant of $\mathsf{MIL}$. In the $\mathsf{TS}$ translation $\Sup$ merged portions of the world, while, in $\mathsf{BSL}$, it merged epistemic possibilities.
Thus while going up in the ordering of $\mathsf{TS}$ means reaching a larger description of the world gaining in content and information, for $\mathsf{BSL}$, the top element in the semilattice  represents  a state of minimal information with all possibilities  still open.  

\newpage

Now, the $\mathsf{TS}$ translation for atoms is `variable-fixed' (modulo bilateralism), $(p)^+=p_+$. From the truth of $p$ at a specific state we cannot draw any conclusion w.r.t. to its truth in higher or lower points in our structure. This gives us the expected  inferential behaviour of a sentence corresponding to an \emph{exact} description of a portion of the world, e.g., $p \wedge q \nvDash p$. In $\mathsf{BSL}$, instead,  atoms $p$ are translated by the modal expression $\PastH\Sast p_{+} $. This  encodes the idea that  atom denotations are ideals in an informational setting. An atomic piece of information is downward closed, persisting in lower elements in the structure   closer  to the state of maximal information at the bottom (but it might become a mere possibility if we go up  to the top state where  nothing is known and everything is  possible). It is  also closed under binary joins, if two sets of epistemic possibilities support an atomic piece of information, so will the merging of these epistemic possibilities, their join. 


\vspace{-1ex}

\subsection{The model transformations}\label{subsec:modeltrans}
Our proof of the target equivalence for faithful translation also involves connecting models.  Intuitively, we have two transformations, one immediate, the other requiring more work:

\vspace{-2.5ex}

\[F:\{t\mid t \text{ is a team}\} \rightleftarrows \{(\mathbb{M},x)\mid (\mathbb{M},x)\text{ is a pointed distr. semilattice model with $x\Vdash \Gamma$}\}:G
\]

\vspace{-3ex}

\paragraph{From team models to semilattice models} First, the obvious map $F$ takes any team model $t$ for $\mathsf{BSL}$ to a pointed distributive semilattice model. As we will see, this map is essentially the identity, since teams as a family of sets of valuations form a semilattice under set-theoretic inclusion and union. For this, we will have

\vspace{-3ex}
\begin{align}
    t&\vDash \varphi\qquad \text{iff}\qquad F(t)\Vdash \varphi^+,\label{eq:adjF}
\end{align}
\vspace{-2.5ex}

\noindent which entails the right-to-left direction of~\eqref{iff} (fullness of the translation), i.e.
\begin{align*}
    \varphi\vDash\psi\qquad \Leftarrow \qquad \Gamma, \varphi^+\Vdash \psi^+
\end{align*}
We do not spell out further details here, since we will only use the map $F$ as a prop to show surjectivity of our crucial model transformation $G$.

\paragraph{From semilattice models to team models}\label{par:faithfulness} The map $G$ runs from pointed distributive semilattice models to teams, and it satisfies the following equivalence, which, as we have seen in Section~\ref{sec:intro}, is characteristic for translations in logic:

\vspace{-3ex}
\begin{align}
        G(\mathbb{M}, x)&\vDash \varphi\qquad \text{iff}\qquad \mathbb{M},x\Vdash \varphi^+\label{eq:adjG}
\end{align} 
\vspace{-3ex}

\noindent This entails the converse, left-to-right direction of~\eqref{iff} (faithfulness of the translation), i.e.

\vspace{-3ex}

\begin{align*}
    \varphi\vDash\psi\qquad \Rightarrow \qquad \Gamma, \varphi^+\Vdash \psi^+
\end{align*}
To see this, assume that $\varphi\vDash \psi$ and that $(\mathbb{M}, x)$ is a pointed distributive semilattice model such that $x\Vdash \Gamma$ and $x\Vdash \varphi^+$. We are then to show that $x\Vdash \psi^+$. By~\eqref{eq:adjG} and $x\Vdash \varphi^+$, we have that the team $G(\mathbb{M},x)\vDash \varphi$. So since $\varphi\vDash \psi$, we get that $G(\mathbb{M},x)\vDash \psi$, and hence~\eqref{eq:adjG} entails that  $x\Vdash \psi^+$, as desired. (It is similarly seen that~\eqref{eq:adjF} implies fullness.\footnote{It is natural to ask if existence of a map $F$ satisfying~\eqref{eq:adjF} not only implies fullness, but is equivalent to it. Likewise for $G$ \& \eqref{eq:adjG} and faithfulness. This is not true in general, but if we generalize the notion of consequence from single-conclusion to multi-conclusion in the obvious way, we have the following characterizations:
\begin{align*}
    \forall\Phi_1,\Phi_2\subseteq \mathsf{Fm(P)}:\Phi_1\vDash \Phi_2 \Leftarrow \Gamma, \Phi_1^+\Vdash \Phi_2^+\qquad &\text{iff} \qquad \text{there is a map $F$ satisfying~\eqref{eq:adjF}}\\
    \forall\Phi_1,\Phi_2\subseteq \mathsf{Fm(P)}:\Phi_1\vDash \Phi_2 \Rightarrow \Gamma, \Phi_1^+\Vdash \Phi_2^+\qquad &\text{iff} \qquad \text{there is a map $G$ satisfying~\eqref{eq:adjG}}
\end{align*}
Thus~\eqref{eq:adjF} precisely corresponds to fullness for multi-conclusion consequence, and~\eqref{eq:adjG} to faithfulness. 
})


\paragraph{Surjectivity of $G$}\label{par:GesssurjF} As noted, we will not prove~\eqref{eq:adjF} directly, but use $F$ as a gadget to prove that $G$ is surjective, by showing that the composition $GF$ is the identity: for all teams $t$, $GF(t)=t$.

This suffices because, more generally, a condition equivalent to that of the existence of a map $F$ satisfying~\eqref{eq:adjF} is to require that $G$ is \textit{essentially surjective}, by which we mean that for all teams $t$, there is a pointed model $(\mathbb{M}, x)$ such that
    \[
        G(\mathbb{M}, x)\vDash \varphi\qquad \text{iff}\qquad t\vDash \varphi.
    \]
    (In some cases, like ours, this may even mean that $G$ is surjective.) Defining a map $F$ using essential surjectivity of $G$, then gives a map satisfying~\eqref{eq:adjF}, as we have
    \[
        t\vDash \varphi \qquad \overset{\text{es. surj.}}{\text{iff}}\qquad G(\mathbb{M}, x)\vDash \varphi\qquad \overset{\eqref{eq:adjG}}{\text{iff}}\qquad \mathbb{M},x\Vdash \varphi^+.
    \]
    Conversely, essential surjectivity of $G$ follows from there being a map $F$ meeting~\eqref{eq:adjF}, as we have 
    \[
        t\vDash \varphi\qquad \overset{\eqref{eq:adjF}}{\text{iff}}\qquad F(t)\Vdash\varphi^+ \qquad\overset{\eqref{eq:adjG}}{\text{iff}}\qquad GF(t)\vDash \varphi
    \]
 for any team $t$.   (Likewise, given a function $F$ satisfying~\eqref{eq:adjF}, there exists a function $G$ satisfying~\eqref{eq:adjG} if and only if $F$ is essentially surjective.)

\vspace{-1.5ex}

\paragraph{The definition of $G$} We define $G$ here and only set out $F$ later when used for surjectivity of $G$. 

\begin{definition}\label{def:pointstovals}
    Given a model $\mathbb{M}$, we say that a point $x$ \emph{defines} a classical valuation $v:\mathsf{P}\to\{0,1\}$ if $\forall p\in\mathsf{P}(x\Vdash p^+ \Leftrightarrow v(p)=1, \; x\Vdash p^- \Leftrightarrow v(p)=0)$.

    We say that $x$ \emph{defines some valuation} if there is a valuation that $x$ defines. Equivalently, if $\forall p\in\mathsf{P}: x\Vdash p^+ \text{ iff } x\nVdash p^-.$ We then write $v_x:\mathsf{P}\to \{0,1\}$ for the valuation so defined.
\end{definition}

\begin{definition}[The model transformation $G$]\label{def:adjG}
    Define 
    \[
        \{t\mid t \text{ is a team}\} \leftarrow \{(\mathbb{M},x)\mid (\mathbb{M},x)\text{ is a pointed distr. semilattice model with $x\Vdash \Gamma$}\}:G
    \]
    by mapping $(\mathbb{M},x)$ to the set of valuations defined by points below $x$, i.e.
    \[
        \{v_y\mid y\leq x\text{ and $y$ defines some valuation}\}.
    \]
    \vspace{-1mm}
    
    More explicitly,
    \[
        G(\mathbb{M},x)\mathrel{:=}\{v\mid \exists y\leq x\;\forall p\in\mathsf{P}(y\Vdash p^+ \Leftrightarrow v(p)=1, \; y\Vdash p^- \Leftrightarrow v(p)=0)\}.
    \]
\end{definition}

\vspace{-1mm}

This is the key mathematical construction of the paper. It is reminiscent of Birkhoff's representation of finite distributive lattices, which sends a lattice-element to the set of \emph{join-irreducibles} below it -- a connection that we will elaborate a bit later on. 

A central property of this map is that it is a \emph{p-morphism} (the terminology comes from p-morphisms in modal logic) in the following sense:

\vspace{-1ex}

\begin{enumerate}[leftmargin=40pt]
        \item [\textnormal{(forth)}] $G(\mathbb{M},x\curlyvee y)=G(\mathbb{M}, x)\cup G(\mathbb{M}, y)$,
        \item [\textnormal{(back)}] if $G(\mathbb{M},x)=s\cup t$, then there are $y, z$ such that $x=y\curlyvee z$, $G(\mathbb{M},y)=s$, and $G(\mathbb{M},z)=t$. 
    \end{enumerate}

\vspace{-1ex}

To prove this, and~\eqref{eq:adjG} for faithfulness thereafter, we state four auxiliary lemmas. To keep the narrative tight, we relegate their proofs to the appendix but indicate their main use. 

\vspace{-2ex}

\subsection{Auxiliary lemmas and the p-morphism property}
 
The first lemma shows that, given any set, $\PastH\Sast$ returns a set closed downwards and under binary joins. The former ensures the right-to-left direction of~\eqref{eq:adjG} for literals, while the latter partly ensures left-to-right as will become clear later. We therefore phrase the lemma in terms of valuations of proposition letters.

\begin{convention}
    Throughout, we use $\pm$ in sub- and superscripts to indicate that the statement applies to both the $+$ and $-$ cases. 
\end{convention}

\begin{lemma}\label{lm:ideals}
        Let $(S, \curlyvee, V)$ be a distributive semilattice model. Then $V(p^\pm)$ is closed downwards and under binary joins.\footnote{We use the notational convention that $V(\varphi)\mathrel{:=}\{x\mid x\Vdash \varphi\}$, so, e.g., $V(p^\pm)=\{x\in S\mid (S, \curlyvee, V),x\Vdash p^\pm\}$. }
\end{lemma}
\begin{proof}
    See the appendix~(\ref{lm:idealsproof}).
\end{proof}

The second lemma says that, under the assumption of $x\Vdash\Gamma$, not only are $V(p^\pm)$ non-empty, but for any point $y\leq x$ there is a point $z\leq y$ that `simulates' the empty team. This is needed for the left-to-right direction of~\eqref{eq:adjG} when disjunction (or negation of conjunction) is present: a team may support a disjunction, say $\varphi\lor \bot$, by virtue of the empty team. It is also convenient to explicitly state here that $V(\Gamma)$ is downward closed -- needed for the above p-morphism claim.

\begin{lemma}
\label{lm:non-emptyideal}
    Let $(\mathbb{M}, x)$ be a pointed distributive semilattice model with $x\Vdash\Gamma$. Then $x\Vdash \PastH\Past \NE^-$, and if $y\leq x$ then $y\Vdash \Gamma$.
\end{lemma}
\begin{proof}
    See the appendix (\ref{lm:non-emptyidealproof}).
\end{proof}
The third auxiliary lemma is best understood as a stepping stone to the fourth lemma following it, which demonstrates how points $x$ can be understood as constituted by the set of valuations defined by points below it.\footnote{One might compare here how Birkhoff's representation of finite distributive lattices relies on (i) every lattice element being the join of join-irreducibles, and (ii) given distributivity, whenever $x=y\curlyvee z$, the join-irreducibles below $x$ are precisely the join-irreducibles that are below $y$ or $z$. Appendix~\ref{app:Birkhoff} strengthens this analogy even further.} 

\begin{lemma}\label{lm:constructionv}
    Let $(\mathbb{M}, x)$ be a pointed distributive semilattice model with $x\Vdash \Gamma$. Then for all $p\in\mathsf{P}$:
    \[
        \text{If } x\nVdash p^+ \text{ and }x\nVdash p^-, \text{ then there are $x^+\curlyvee x^-= s$ such that $x^\pm\Vdash p^\pm$ and $x^\pm\nVdash p^\mp$.}
    \]
\end{lemma}
\begin{proof}
    See the appendix (\ref{lm:constructionvproof}).
\end{proof}

\begin{lemma}\label{lm:joinofvals}
    Let $(\mathbb{M}, x)$ be a pointed distributive semilattice model with $x\Vdash \Gamma$. If $x\Vdash \NE^+$, then there is a non-empty finite set of points $S_x$ such that 
    \begin{itemize}
        \item $x=\bigcurlyvee S_x$,
        \item each $y\in S_x$ defines a valuation, i.e. $\forall p\in\mathsf{P}:y\Vdash p^+ \text{ iff } y\nVdash p^-$,
        \item $G(\mathbb{M}, x)=\{v_{y}\mid y\in S_x\}$.
    \end{itemize}
    Moreover, whenever $S_x'$ is a non-empty finite set satisfying the first two items, it also satisfies the third.
\end{lemma}
\begin{proof}
    See the appendix (\ref{lm:joinofvalsproof}).
\end{proof}

With these lemmas in hand, that $G$ is a p-morphism (and more) can be proven.

\begin{lemma}[forth]\label{lm:forthG}
    $G$ is a homomorphism. That is, for all pointed distributive semilattice models $(\mathbb{M}, x)$ with $x\Vdash \Gamma$, and all $y,z$ such that $x=y\curlyvee z$, 
    \[
        G(\mathbb{M},x)=G(\mathbb{M}, y)\cup G(\mathbb{M}, z)
    \]
\end{lemma}
\begin{proof}
    First note that, as satisfaction of $\Gamma$ is closed downwards (cf. Lemma~\ref{lm:non-emptyideal}), $G(\mathbb{M}, y)$ and $G(\mathbb{M}, z)$ are well-defined. 

    We prove $G(\mathbb{M},y\curlyvee z)=G(\mathbb{M}, y)\cup G(\mathbb{M}, z)$ by cases.

    If $y\Vdash \NE^-$ and $z\Vdash \NE^-$, then $x=y\curlyvee z\Vdash \NE^-$ because $V(\NE^-)=\bigcap_{p\in \mathsf{P}}(V(p^+)\cap V(p^-))$ is closed under binary joins. Consequently, 
    \(
        G(\mathbb{M},x)=\varnothing=\varnothing\cup \varnothing=G(\mathbb{M}, y)\cup G(\mathbb{M}, z).
    \)

    If $y\Vdash \NE^-$ and $z\nVdash \NE^-$, then $x=y\curlyvee z\nVdash \NE^-$ as $V(\NE^-)$ is downward closed. Thus $z\Vdash \NE^+$ and $x\Vdash \NE^+$, and so, because of the former, the preceding lemma supplies a non-empty finite set $S_z$ with $z=\bigcurlyvee S_z$, each $z'\in S_z$ defining a valuation, and $G(\mathbb{M},z)=\{v_{z'}\mid z'\in S_z\}$. We claim that $S_x'\mathrel{:=}\{y\curlyvee z'\mid z'\in  S_z\}$ satisfies the first two items of the preceding lemma for $x$, and therefore the third. It satisfies the first because $x=y\curlyvee z=y\curlyvee \bigcurlyvee S_z=\bigcurlyvee S_{x'}$; and the second because $S_z$ does and because for each $z'\in S_z$,
    \begin{alignat*}{3}
        y\curlyvee z'\Vdash p^\pm &\quad \Rightarrow &\quad z'\Vdash p^\pm\tag{by downward closure}\\
        y\curlyvee z'\Vdash p^\pm &\quad \Leftarrow &\quad z'\Vdash p^\pm\tag{by $y\Vdash \NE^-$, so $y\Vdash p^\pm$, and binary-join closure}
    \end{alignat*}
    It follows that $G(\mathbb{M},x)=\{v_{y\curlyvee z'}\mid z'\in S_z\}=\{v_{z'}\mid z'\in S_z\}=\varnothing\cup G(\mathbb{M},z)= G(\mathbb{M},y)\cup G(\mathbb{M},z)$.

    Similarly if $y\nVdash \NE^-$ and $z\Vdash \NE^-$. Lastly, if $y\nVdash \NE^-$ and $z\nVdash \NE^-$, then $x,y,z\Vdash \NE^+$, and the preceding lemma supplies $S_y$ and $S_z$ for $y$ and $z$, respectively. It follows immediately that $S_x'\mathrel{:=}S_y\cup S_z$ satisfies the first two items of the preceding lemma for $x$, and therefore the third. Hence \(
        G(\mathbb{M},x)=\{v_{x'}\mid x'\in S_x'\}=\{v_{y'}\mid y'\in S_y\}\cup \{v_{z'}\mid z'\in S_z\}=G(\mathbb{M}, y)\cup G(\mathbb{M}, z).
    \)
\end{proof}

\begin{lemma}[back]\label{lm:backG}
$G$ satisfies the Back condition. That is, for all pointed distributive semilattice models $(\mathbb{M}, x)$ with $x\Vdash \Gamma$, and all teams $s,t$ such that $G(\mathbb{M},x)=s\cup t$, there are $y, z$ such that 
    \[
       x=y\curlyvee z, \qquad G(\mathbb{M},y)=s,\qquad G(\mathbb{M},z)=t.
    \]
\end{lemma}

\begin{proof}
    If $s=\varnothing$, then $t=G(\mathbb{M},x)$. Besides, by Lemma~\ref{lm:non-emptyideal}, we have $x\Vdash \Past\NE^-$, so there is $y\leq x$ with $y\Vdash \NE^-$. Letting $z\mathrel{:=}x$, gives us the required. Similarly if $t=\varnothing$. 
    
    So suppose $s\neq \varnothing\neq t$. Then $x\Vdash \NE^+$ and the preceding lemma is applicable, providing a non-empty finite set $S_x$. 
    Let $y\mathrel{:=}\bigcurlyvee\{w\in S_x\mid v_{w}\in s\}$ and $z\mathrel{:=}\bigcurlyvee\{w\in S_x\mid v_{w}\in t\}$, and note that the joins are well-defined, as both sets are finite and non-empty (non-empty because $s\neq \varnothing\neq t$). Moreover, $S_x=\{w\in S_x\mid v_{w}\in s\}\cup \{w\in S_x\mid v_{w}\in t\}$, since $\{v_{w}\mid w\in S_x\}=G(\mathbb{M},x)=s\cup t$, whence $x= \bigcurlyvee S_x= y\curlyvee z$. By the uniqueness-part of the previous lemma for $y$ and $\{w\in S_x\mid v_{w}\in s\}$, we obtain $G(\mathbb{M},y)=s$; and similarly we get $G(\mathbb{M},z)=t$,  completing the proof.
\end{proof}

\vspace{-2ex}

\subsection{Faithfulness of the translation}

\begin{lemma}[\ref{eq:adjG}]\label{lm:adjG}
    For all formulas $\varphi\in\mathsf{Fm}(\mathsf{P})$ and pointed distributive semilattice models $(\mathbb{M},x)$ with $x\Vdash \Gamma$, we have the following equivalence:
    \begin{align*}
        G(\mathbb{M}, x)&\vDash \varphi\qquad \text{iff}\qquad \mathbb{M},x\Vdash \varphi^+
\end{align*} 
\end{lemma}
\begin{proof}
    Let $(\mathbb{M},x)$ be an arbitrary distributive semilattice models with $x\Vdash \Gamma$, and recall from Definition~\ref{def:adjG} that $G(\mathbb{M},x)$ is the set of valuations
    \[
        \{v\mid \exists y\leq x\;\forall p\in\mathsf{P}(y\Vdash p^+ \Leftrightarrow v(p)=1, \; y\Vdash p^- \Leftrightarrow v(p)=0)\}.
    \]
    We prove that for all $\varphi\in\mathsf{Fm}(\mathsf{P})$ and $y\leq x$,
    \begin{align*}
        G(\mathbb{M},y)\vDash \varphi\qquad \text{iff}\qquad \mathbb{M},y\Vdash \varphi^+,\\
        G(\mathbb{M},y)\Dashv\varphi \qquad \text{iff}\qquad \mathbb{M},y\Vdash \varphi^-.
    \end{align*}
    (Recall that $V(\Gamma)$ is downward closed (as per Lemma~\ref{lm:non-emptyideal}), so $G(\mathbb{M},y)$ is well-defined for $y\leq x$.)
    
    The proof is by induction on $\varphi$. 
    
    For proposition letters $p\in\mathsf{P}$, we have:
    
    \vspace{-1ex}
    
    \[
        G(\mathbb{M},y)\vDash p\qquad \overset{\text{(def)}}{\text{iff}} \qquad \forall v\in G(\mathbb{M},y)(v(p)=1)\qquad \overset{(i)}{\text{iff}}\qquad \mathbb{M},y\Vdash p^+.
    \]
    
    \vspace{1mm}
    
    Here, right-to-left of $(i)$ holds since $V(p^+)$ 
    is a downset, and left-to-right since it is closed under binary joins. While the former is clear, for the latter, assume by way of contraposition that $y\nVdash p^+$. Then $y\nVdash \NE^-$, hence $y\Vdash \NE^+$, so Lemma~\ref{lm:joinofvals} provides a non-empty finite set $S_y$ with $y=\bigcurlyvee S_y$, each $y'\in S_y$ defining a valuation, and $G(\mathbb{M},y)=\{v_{y'}\mid y'\in S_y\}$. Closure of $V(p^+)$ under binary joins then implies that for some $y'\in S_y$, $y'\nVdash p^+$, and so $v_{y'}(p)=0$, as desired.
    
    Similarly, we get $G(\mathbb{M},y)\Dashv p$ iff $\mathbb{M},y\Vdash p^-$.

    Next, for $\NE^+$ we have
    \[
        G(\mathbb{M},y)\vDash \NE\quad \overset{\text{(def)}}{\text{iff}}\quad G(\mathbb{M},y)\neq \varnothing \quad \overset{(ii)}{\text{iff}}\quad \mathbb{M},y\Vdash \bigwedge_{p\in \mathsf{P}}{\sim}(p^+\land p^-)\quad \overset{\text{(def)}}{\text{iff}}\quad \mathbb{M},y\Vdash \NE^+.
    \]
    
    \vspace{-1ex}
    
    Right-to-left of $(ii)$ follows by Lemma~\ref{lm:joinofvals}. 
    Conversely, if there is $p\in\mathsf{P}$ such that $y\Vdash p^+\land p^-$, then we also have $z\Vdash p^+\land p^-$ for all $z\leq y$, hence $G(\mathbb{M},y)=\varnothing$. 

    Using this, we also obtain
    
    \vspace{-0.5ex}
    
    \[
        G(\mathbb{M},y)\Dashv \NE \quad \text{iff}\quad \mathbb{M},y\Vdash \NE^-,
    \]
    
        \vspace{-1ex}
        
    \noindent since we have that
    
            \vspace{-3.5ex}
    \begin{align*}
        G(\mathbb{M},y)\Dashv \NE\quad &\text{iff}\quad G(\mathbb{M},y)\nvDash \NE,\\
        \mathbb{M},y\nVdash \NE^+ \quad &\text{iff}\quad \mathbb{M},y\Vdash\NE^-,
    \end{align*}
    
            \vspace{-1ex}
    \noindent where the former holds by definition and the latter because $y\Vdash \PastH(\NE^+\lor \NE^-)$, so we know that $y\Vdash \NE^+\lor \NE^-$, and also $y\nVdash \NE^+\land \NE^-$.   

    The last base case, $\bot$, is then immediate.

    Among the inductive steps, only the support cases for diamond $\blacklozenge$ and disjunction $\lor$ (as well as the analogous antisupport case for conjunction $\land$) are non-trivial. Since a team supports $\blacklozenge\varphi$ iff it supports $(\varphi\land \NE)\lor \neg\bot$, and a world satisfies $(\blacklozenge\varphi)^+$ iff it satisfies $[(\varphi\land \NE)\lor \neg\bot]^+$, the $\blacklozenge$-support case will follow if we establish the $\lor$-support case. Thus it remains only to establish the support case for disjunction (again, the antisupport case for conjunction is proven entirely analogously). And for this, it suffices simply that $G$ is a p-morphism.
    
    The forth condition (or homomorphic condition) ensures that
    \[
        G(\mathbb{M},y)\vDash \varphi\vee\varphi'\quad \Leftarrow \quad \mathbb{M},y\Vdash (\varphi\vee\varphi')^+ ,
    \]
    and the back condition that
    \[
        G(\mathbb{M},y)\vDash \varphi\vee\varphi' \quad \Rightarrow \quad \mathbb{M},y\Vdash (\varphi\vee\varphi')^+.
    \]
    This follows by a standard argument, but for completeness, we include it here. 
    
    First, assume that $y\Vdash (\varphi\lor \varphi')^+$, i.e. $y\Vdash \Sup \varphi^+ \varphi'^+$. Then there are $z,z'$ s.t. $z\Vdash \varphi^+, z'\Vdash \varphi'^+$, and $y=z\curlyvee z'$. By the induction hypothesis, $G(\mathbb{M},z)\vDash \varphi$ and $G(\mathbb{M},z')\vDash \varphi'$, and by (forth), $G(\mathbb{M},y)=G(\mathbb{M},z)\cup G(\mathbb{M},z')$, whence $G(\mathbb{M},y)\vDash \varphi\vee\varphi'$. Conversely, assume that $G(\mathbb{M},y)\vDash \varphi\lor \varphi'$. Then there are teams $s, t$ s.t. $s\vDash \varphi$, $t\vDash \varphi'$, and $G(\mathbb{M},y)=s\cup t$. By (back), there are $z, z'$ s.t. $y=z\curlyvee z'$, $G(\mathbb{M},z)=s$, $G(\mathbb{M},z')=t$, which by the induction hypothesis yields $y=z\curlyvee z'\Vdash \Sup\varphi^+\varphi'^+$, as desired.
\end{proof}

\subsection{Fullness of the translation}

We are left to formally define $F$ to show that $GF(t)=t$, and hence that $G$ is surjective. 

\begin{definition}[The model transformation $F$]\label{def:adjF}
    Define 
    \[
        F:\{t\mid t \text{ is a team}\} \to \{(\mathbb{M},x)\mid (\mathbb{M},x)\text{ is a pointed distr. semilattice model with $x\Vdash \Gamma$}\}
    \]
    by mapping a team $t$ to the frame $(\mathcal{P}(t), \cup)$---which as a powerset is a distributive semilattice---pointed at $t$ and with valuation $V:\mathsf{P}_\pm\to \mathcal{P}\mathcal{P}(t)$ given by
    \begin{align*}
        &V(p_+)\mathrel{:=}\{s\in \mathcal{P}(t)\mid \text{for all $v\in s:v(p)=1$}\}={\downarrow}\{v\in t\mid v(p)=1\},\\
        &V(p_-)\mathrel{:=}\{s\in \mathcal{P}(t)\mid \text{for all $v\in s:v(p)=0$}\}={\downarrow}\{v\in t\mid v(p)=0\}.
    \end{align*}
\end{definition}
To prove that $GF$ is the identity, we need a lemma showing that $\PastH\Sast$ is a fixpoint operator on ideals over powersets; for if not,~\eqref{eq:adjF} (implied by $GF$ being the identity) would fail for literals. 
We prove a slightly more general result and, for convenience, 
phrase it in terms of valuations of proposition letters.
\begin{lemma}\label{lm:closurevalideals}
    Let $(S, \curlyvee, V)$ be a distributive semilattice model and $p_\pm\in \mathsf{P}_\pm$. If $V(p_\pm)$ is an ideal, 
    then for all $x\in S$,
    \[ 
        x\Vdash p_\pm\quad \text{iff} \quad x\Vdash p^\pm
    \]
\end{lemma}
\begin{proof}
    If $V(p_\pm)$ is an ideal, then
    \[
        x\Vdash p_\pm \quad \text{iff}\quad x\Vdash\PastH\Sast p_\pm\quad \overset{\text{(def)}}{\text{iff}} \quad x\Vdash p^\pm,
    \]
    where right-to-left of the first `iff' is a consequence of reflexivity, so $x\leq x$, and ideals being closed under (non-empty) finite joins.
\end{proof}

The next lemma then gives the desired by establishing that (i) $F$ is well-defined, i.e. $F(t)\Vdash \Gamma$, and then (ii) for all teams $t$, $t=GF(t)$.
\begin{lemma}[\ref{eq:adjF}]\label{lm:adjF}
    For all teams $t$, $F(t)\Vdash \Gamma$ and $t=GF(t)$.
\end{lemma}
\begin{proof}
    See the appendix~(\ref{lm:adjFproof}). 
\end{proof}

Our main theorem now follows.

\begin{theorem}\label{th:translation}
    For all $\varphi, \psi\in \mathsf{Fm}(\mathsf{P})\,{:}\quad\varphi\vDash\psi\quad \text{iff} \quad \Gamma, \varphi^+\Vdash \psi^+.$
\end{theorem}
\begin{proof}
    By Lemma~\ref{lm:adjG}, we obtain
    \begin{align*}
    \varphi\vDash\psi\qquad \Rightarrow \qquad \Gamma, \varphi^+\Vdash \psi^+,
\end{align*}
    and by Lemma~\ref{lm:adjF}, the converse
    \begin{align*}
    \varphi\vDash\psi\qquad \Leftarrow \qquad \Gamma, \varphi^+\Vdash \psi^+
\end{align*}
(recall, respectively, the second and third paragraphs of Section~\ref{subsec:modeltrans}).
\end{proof}

\subsection{Translating further team logics}

Our translation from $\mathsf{BSL}$ into modal information logic allows for many variations. In the background here is the fact that models and meanings for a logical system (whether classical modal or hyperintensional) come about by fixing three parameters: the \textit{frames}, the \textit{valuations}, and the \textit{clauses} for the connectives. It may even be said that the second (with the first) give the meaning of proposition letters, and the third (with the first) that of the connectives. In line with this,  the modular set-up given above can be modified in various ways.

\medskip



This flexibility means that our translation also applies to other team logics. It can be extended to include the inquisitive disjunction $\vvee$ and implication $\to$, as well as the Boolean negation ${\sim}$. This way, we obtain a faithful translation of inquisitive logic and of $\mathsf{PL}({\sim})$, along with propositional dependence logic, independence logic, inclusion logic, and exclusion logic.\footnote{For more on inquisitive logic, see~(\cite{CiardelliGroenendijkRoelofsen2018, Ciardellibook2022}); and for more on $\mathsf{PL}({\sim})$ and on propositional dependence-style logics, see e.g.~(\cite{YangVäänänen2016, YangVäänänen2017,Yang2022}).}

We now consider different grammars over our fixed finite set of proposition letters $\mathsf{P}$.

\vspace{-2ex}

\paragraph{Inquisitive logic} The inquisitive propositional grammar is given by
\[
    \varphi\mathrel{::=} \bot\hspace{0.1cm}|\hspace{0.1cm} p \hspace{0.1cm}| \hspace{0.1cm} \varphi\land\varphi\hspace{0.1cm}|\hspace{0.1cm}\varphi\to\varphi\hspace{0.1cm}|\hspace{0.1cm}\varphi\vvee\varphi
\]
As with $\mathsf{BSL}$, inquisitive formulas are evaluated at sets of valuations $t\subseteq \{v\mid v:\mathsf{P}\to\{0,1\}\}$, but according to the following unilateral clauses:
\begin{align*}
        &t\vDash p && \text{iff} && \text{for all } v\in t, v(p)=1\\
        &t\vDash \bot && \text{iff} && t=\varnothing\\
        &t\vDash \varphi\land\psi && \text{iff} && t\vDash \varphi \text{ and } t\vDash \psi\\
        &t\vDash \varphi\to\psi && \text{iff} && \text{for all } t'\subseteq t, \text{ if }t'\vDash \varphi\text{ then }t'\vDash \psi\\
        &t\vDash \varphi\vvee\psi && \text{iff} && t\vDash \varphi \text{ or } t\vDash \psi
\end{align*}
The previous translation $\cdot^+$ is extended to $\vvee$ and $\to$ by setting
\begin{align*}
    (\varphi\vvee \psi)^+&\mathrel{:=}\varphi^+\vee\psi^+\\
    (\varphi\to \psi)^+&\mathrel{:=} \PastH(\varphi^+\to \psi^+)
\end{align*}
It is  easy  to check that the proofs of the two directions of Theorem~\ref{th:translation}, along with their auxiliary lemmas, also go through in the inquisitive setting. Thus, for all inquisitive formulas $\varphi, \psi$,
\begin{align*}
    \varphi\vDash\psi\qquad \text{iff} \qquad \Gamma, \varphi^+\Vdash \psi^+
\end{align*}


\paragraph{Propositional team logic with Boolean negation ($\mathsf{PL}({\sim})$)} As is standard in the team-semantical literature, the grammar of $\mathsf{PL}({\sim})$ is presented in negation-normal form as follows.
\[
    \varphi\mathrel{::=}  p \hspace{0.1cm}| \hspace{0.1cm} \neg p \hspace{0.1cm}| \hspace{0.1cm} \varphi\land\varphi\hspace{0.1cm}|\hspace{0.1cm}\varphi\lor\varphi\hspace{0.1cm}|\hspace{0.1cm} {\sim}\varphi.
\]

The clauses are unilateral, and those for $p, \neg p, \varphi\land\varphi, \varphi\lor\varphi$ are as in $\mathsf{BSL}$, while ${\sim}$ is now classical Boolean negation:
\begin{align*}
        t\vDash {\sim}\varphi \qquad \text{iff} \qquad t\nvDash \varphi.
\end{align*}

We extend $\cdot^+$ by setting
\begin{align*}
    ({\sim}\varphi)^+&\mathrel{:=}{\sim} \varphi^+.
\end{align*}


Once again, it is readily verified that Theorem~\ref{th:translation} still holds. Thus, we get that for all formulas $\varphi, \psi$ in the language of Propositional Team Logic with Boolean negation,
\begin{align*}
    \varphi\vDash\psi\qquad \text{iff} \qquad \Gamma, \varphi^+\Vdash \psi^+.
\end{align*}


\paragraph{More propositional team logics} It is now a small step to see that the translation extends across numerous propositional team logics, including those contained in the following grammar where various dependency atoms are included:
\[
    \varphi\mathrel{::=} \bot\hspace{0.1cm}|\hspace{0.1cm} \NE\hspace{0.1cm}|\hspace{0.1cm}p \hspace{0.1cm}| \hspace{0.1cm} \neg p\hspace{0.1cm}|\hspace{0.1cm}\varphi\lor\varphi\hspace{0.1cm}|\hspace{0.1cm}\varphi\land\varphi\hspace{0.1cm}|\hspace{0.1cm}\blacklozenge \varphi\hspace{0.1cm}|\hspace{0.1cm}\varphi\vvee\varphi\hspace{0.1cm}|\hspace{0.1cm}\varphi\to\varphi\hspace{0.1cm}|\hspace{0.1cm}{\sim}\varphi\hspace{0.1cm}|\hspace{0.1cm}{=}(\vec{\alpha};\vec{\alpha})\hspace{0.1cm}|\hspace{0.1cm}\vec{\alpha}\bot_{\vec{\alpha}}\vec{\alpha}\hspace{0.1cm}|\hspace{0.1cm}\vec{\alpha}\subseteq\vec{\alpha}\hspace{0.1cm}|\hspace{0.1cm}\vec{\alpha}|\vec{\alpha}\hspace{0.1cm}|\hspace{0.1cm}\vec{\alpha}\Upsilon\vec{\alpha},
\]
with $\alpha$ free from the new connectives and dependency atoms $\vvee, {\sim}, {=}(\cdot, \cdot), \hdots$ (and $\blacklozenge$ and $\NE$).
%

All these atoms are definable in the team language with Boolean negation, ${\sim}$ (cf.~\cite[page~8]{LueckVilander19:LMCS}), and so they can be translated accordingly.

It  follows that for all formulas $\varphi, \psi$ in this extended language,
\begin{align*}
    \varphi\vDash\psi\qquad \text{iff} \qquad \Gamma, \varphi^+\Vdash \psi^+
\end{align*}

\noindent{\bf On which models do these translations work?}   With these translations established, one may wonder over which model classes for these logics our translation will continue to work. We just make one observation to this effect. As the proof of the right-to-left direction of Theorem~\ref{th:translation} provides a powerset frame and powerset frames are distributive semilattices, it follows that the translation is full and faithful both when the modal frames are powersets and when they are distributive semilattices. Thus, also for any class of frames in between, including distributive semilattices with a least element, distributive lattices, Heyting algebras, or Boolean algebras.\footnote{If, however, we give up the requirement of distributivity, and consider, e.g., arbitrary semilattices, then the translation is no longer faithful: $p^\pm$ need not be an ideal, as witnessed by the diamond lattice $M_3=\{0,x,y,z,1\}$ when setting $V(p_+)=\{0,x,y\}$. Because then $V(\Sast p_+)=\{0,x,y,1\}$ and $V(p^+)=V(\PastH\Sast p_+)=\{0,x,y\}$, so $V(p^+)$ is not closed under joins, which means that, e.g., translated $\lor$-idempotence ($\Gamma,\Sup p^+ p^+\Vdash p^+$) can fail.}

\begin{remark}[Translating modal extensions] Our translation addressed the 'propositional engine' underlying logics like $\mathsf{BSL}$ or Fine's truthmaker logic. The modalities employed here can be viewed as a sort of `inner modalities' driving reasoning with propositional core notions such as conjunction and disjunction. But in the literature, there are also so-called `modal extensions' of these hyperintensional systems. 

For instance, the clauses for an added relational modality $\Diamond$ in $\mathsf{BSL}$ to obtain the full logic $\mathsf{BMSL}$ run as follows (with $M=\{W,R,V\}$ and $R[w]=\{v \in W \mid wRv\}$): 
   \begin{eqnarray*}
M,  s \vDash   \Diamond  \varphi   &\mbox{ iff }&    {\mbox{for all }   w\in   s:   \exists t \subseteq R[w]: t \neq \varnothing  \ \& \ M,t \vDash \varphi}  \\ 
M, s \Dashv   \Diamond  \varphi   &\mbox{ iff }&   \mbox{for all }  w \in  s:   M, R[w]  \Dashv \varphi      
   \end{eqnarray*}

These 'outer modalities' model explictly modal notions with one of their classical interpretations: alethic, epistemic, deontic, and so on. 

We believe that our translation extends to this case, now  by translating into a modal logic with one additional modality. We will not go further into this matter here.
\end{remark}

\section{Repercussions}\label{sec:repercussions}

\paragraph{Hyperintensional logics} Our translation suggests various issues concerning the hyperintensional logics that we have been looking at. 

\vspace{1ex}

For instance, we isolated the structural features of powersets needed for `team-semantical reasoning': we go from powersets to distributive semilattices. This suggests that the translation may work as a gateway to abstract (generalized) team semantics, as advocated by \textcite{BezhanishviliYang2024}. 

\newpage
Another insight emanating from our analysis concerns computational complexity. Since all the team logics considered here are decidable, so are their translated images, which are fragments of the target modal logic(s). This is intriguing, as the target modal logic(s) -- be it over distributive semilattices, distributive lattices, Heyting algebras, Boolean algebras or powersets -- are undecidable. This follows from~\textcite{Knudstorp2026b}, as there undecidability is established already in the language without $\Sast$. 
We have already discussed in Section~\ref{subsec:comparison} what this suggests for studying the complexity-raising role of classical negation, or for a view of hyperintensional logics as positive fragments of classical modal logics.

\vspace{1ex}

The setting of modal information logics may also suggest the design of richer hyperintensional logics. For instance, one basic feature of modal logics like $\mathsf{MIL}$ is that the existential `merge modality' distributes over classical disjunctions in both its arguments. Given that disjunction has been interpreted differently in many logics, classical, team-based, or with `eventual choice' in possibility semantics (see~\cite{Holliday2021}), it might be of interest to create one logic where all these variants live and their interactions can be studied. The classical modal distribution law just mentioned would then say that, for instance, a $\mathsf{BSL}$ disjunction distributes over classical disjunctions -- but this will not hold the other way around.

\vspace{1ex}
One might also say that, by placing a whole family of hyperintensional logics inside the same classical modal system, we are making these logics more comparable among themselves. In our particular case, preorders and semilattices are widely used models for information structure, so all translated hyperintensional logics now live on the same models, but they differ in which features of propositions in these models they choose to focus on.\footnote{Other ways of doing this include the unified `state-support semantics' in~\cite{MassasPuncocharForthcoming}.}


\vspace{-2ex}

\paragraph{Classical modal logic} On the classical  side, we found the need for two extended systems that do not seem to have been studied so far: modal information logic with a $\mu$-calculus extension $\mu\mathsf{MIL}$ and modal $\mathsf{WSOL}$. Both systems raise interesting issues about different ways of extending modal logics into the infinite, independently from the particular use made of them here.

The fixed-point logic $\mu\mathsf{MIL}$, or its `continuous fragment' with fixed-points that always stabilize by stage $\omega$ extend known modal fixed-point logics in natural ways. Do the core features of the modal $\mu$-calculus carry over here, i.e., decidability and axiomatizability essentially by just adding the principles defining pre-fixed-points? Can we prove analogues of famous semantic characterizations such as the Janin-Walukiewicz Theorem?\footnote{Given our translation, a system of further interest would be, not the minimal one over ternary relations, but over semilattices. In that case the logic is undecidable, as we noted, but this would still be a natural specialization.}

The alternative weak second-order version could be considered a natural modal fragment of the full system $\mathsf{WSOL}$ which is defined over the complete first-order language. Surprisingly little is known about the latter infinitary logic, except for some obvious features like invariance under potential isomorphisms. Its modal fragments (one could think of allowing arbitrary guarding, with the condition for the binary merge modality as a special case) also seem natural. To mention just one question, is guarded $\mathsf{WSOL}$ decidable?

\vspace{-2ex}

\paragraph{General features of translations} What drives our new translation is formalization of hyperintensional truth conditions in modal terms (like Gödel did in his analysis of intuitionistic logic) plus taking care of atoms by special assumptions. 
This leaves room for many variations. Translation of logical operations gives a core connection, but translations also suggest freedom of system design in the other parameters such as the treatment of proposition letters or the auxiliary assumptions.\footnote{The compositional  `logical operator transform' at the heart of a translation can also be studied by itself. For instance, \textcite{Benthem1986} determines which sublogic of $\mathsf{S4}$ arises under the $G$-translation of intuitionistic formulas if we just send proposition letters to themselves, allowing for some non-persistent behavior.} 

\vspace{1ex}

Also, with a translation in hand, there is always the natural question whether it can be converted. We already noted that, in the opposite direction of the Gödel translation, there is also a faithful translation of modal $\mathsf{S4}$ into intuitionistic logic. Are there also such converse translations from classical modal logics into hyperintensional logics? Here the complexity difference noted earlier (undecidable versus decidable) seems an insuperable barrier.

\vspace{1ex}

One can also view our new translation as one more case study in the general area of translations, as surveyed in~(\cite{French2017}). In particular, while our focus was on connecting specific hyperintensional and classical logics, one could also shift the perspective as follows. Viewed abstractly, translation is really a \emph{ternary relation} 

\vspace{-2ex}

$$L_1, \,\, \tau, \, \, L_2$$ 

\vspace{-1ex}

\noindent where $L_1$ is a source logic, $L_2$ a target logic, and $\tau$ some syntactic translation. In this scheme one can fix any of the three arguments and ask for the possible variation in the remaining ones. For instance, there might be other translations, different from ours, running from $\mathsf{BSL}$ into $\mu\mathsf{MIL}$. Or, fixing our translation, there can be many more source and target logics connected by it, perhaps on much weaker propositional bases, a phenomenon already noted in connection with the Gödel translation in~(\cite{HollidayMassas2026}).

\medskip
One further general topic not addressed in this paper is just  which technical properties of logical systems are transferred by translations. For some discussion of this issue see~(\cite{Benthem2026}). For instance, while decidability and axiomatizability were in the background in the above, what about more subtle features such as interpolation theorems expressing some sort of expressive completeness of the chosen language for a logical system?

\medskip	

\noindent{\bf Other methodologies for studying translations} Finally, our analysis of connections between hyperintensional and classical logics was semantic, and in fact model-theoretic, in line with their usual style of presentation. But there are  broad alternatives. One could study connections between our systems \emph{algebraically} using tools from Universal Algebra. In fact our use of Birkhoff's representation for finite distributive lattices already went some way in this direction. A further style of analysis might be \emph{proof-theoretic}, concentrating more on connecting the natural proof structure of hyperintensional and classical modal logics.\vspace{-2ex}

\section{Summary}\label{sec:summary}	
We have presented a new translation from team-based hyperintensional logics to classical modal ones extending an earlier `bridge head' connecting truthmaker logics to modal information logics. Of course, this is not an issue of `reduction' or making one side in the contact redundant. As we saw, our translation makes systems from the two sides more comparable, and it also suggests interesting new questions about their design. More generally, this paper can be seen as an exercise in finding what connects us, rather than what divides us, in the field of logic.

\vspace{-1ex}

\paragraph{Acknowledgements.} Part of this research was supported by the NWO OC project \textit{Nothing is Logical}  (grant no 406.21.CTW.023). We would also like to thank the audiences of a MCMP colloquium and a Tsinghua seminar for feedback on versions of this material.

\appendix
\section{Appendix}
\subsection{Four auxiliary lemmas for faithfulness}
We prove 
the four auxiliary Lemmas~\ref{lm:ideals}--\ref{lm:joinofvals}.

\begin{lemma}[Restatement of Lemma~\ref{lm:ideals}, with proof]\label{lm:idealsproof}
        Let $(S, \curlyvee, V)$ be a distributive semilattice model. Then $V(p^\pm)$ is closed downwards and under binary joins.
\end{lemma}

\begin{proof}
    By transitivity of $\leq$, $V(p^\pm)=V(\PastH\Sast p_\pm)$ is downward closed. To see that it is closed under binary joins, assume $t\Vdash p^\pm$ and $u\Vdash p^\pm$, and let $x\leq t\curlyvee u$. We show that $x\Vdash \Sast p_\pm$. By distributivity, there are $t'\leq t$ and $u'\leq u$ such that $x=t'\curlyvee u'$. Since $t'\leq t\Vdash p^\pm$, we have $t'\Vdash \Sast p_\pm$, so there are $t'_1, \hdots, t'_n$ such that $t'=t_1'\curlyvee \cdots \curlyvee t_n'$ and $t_i'\Vdash p_\pm$ for all $i\in\{1, \hdots, n\}$. Likewise, there are $u'_1, \hdots, u'_m$ such that $u'=u_1'\curlyvee \cdots \curlyvee u_m'$ and $u_j'\Vdash p_\pm$ for all $j\in\{1, \hdots, m\}$. But then $x=t'\curlyvee u'=t_1'\curlyvee \cdots \curlyvee t_n'\curlyvee u_1'\curlyvee \cdots \curlyvee u_m'$, so $x\Vdash \Sast p_\pm$. As $t,u,$ and $x$ were arbitrary, this proves that $V(p^\pm)$ is closed under binary joins.
\end{proof}
\begin{lemma}[Restatement of Lemma~\ref{lm:non-emptyideal}, with proof]
\label{lm:non-emptyidealproof}
    Let $(\mathbb{M}, x)$ be a pointed distributive semilattice model with $x\Vdash\Gamma$. Then $x\Vdash \PastH\Past \NE^-$, and if $y\leq x$ then $y\Vdash \Gamma$.
\end{lemma}
\begin{proof}
    For the latter, by the preceding lemma, $V(p^+)$ and $V(p^-)$ are downward closed, hence by distributivity of the semilattice, so is $V(\Sup p^+p^-)$ (and therefore $V(\Gamma)$ too). 

    For the former, it is then readily seen that for all $y\leq x$ there is $z\leq y$ with $z\Vdash \bigwedge_{p\in \mathsf{P}}p^+\land p^-$, i.e. $x\Vdash \PastH\Past \NE^-$, as desired.
\end{proof}
\begin{lemma}[Restatement of Lemma~\ref{lm:constructionv}, with proof]\label{lm:constructionvproof}
    Let $(\mathbb{M}, x)$ be a pointed distributive semilattice model with $x\Vdash \Gamma$. Then for all $p\in\mathsf{P}$:\vspace{-0.7ex}
    \[
        \text{If } x\nVdash p^+ \text{ and }x\nVdash p^-, \text{ then there are $x^+\curlyvee x^-= s$ such that $x^\pm\Vdash p^\pm$ and $x^\pm\nVdash p^\mp$.}
    \]
\end{lemma}
\begin{proof}
    Assume $x\nVdash p^+ \text{ and }x\nVdash p^-$ for some $p\in\mathsf{P}$. 
     Since $x\Vdash \Gamma$, we have $x\Vdash \Sup p^+p^-$. 
     So there are $x^+,x^-$ s.t. $x^+\Vdash p^+, x^-\Vdash p^-$, and $x=x^+\curlyvee x^-$. Moreover, $x^+\nVdash p^-$, since had it been that $x^+\Vdash p^-$, then because $V(p^-)$ is closed under binary joins (cf. Lemma~\ref{lm:ideals}), we would have $x=x^+\curlyvee x^-\Vdash p^-$ -- contradiction. Similarly, $x^-\nVdash p^+$.
\end{proof}
\begin{lemma}[Restatement of Lemma~\ref{lm:joinofvals}, with proof]\label{lm:joinofvalsproof}
    Let $(\mathbb{M}, x)$ be a pointed distributive semilattice model with $x\Vdash \Gamma$. If $x\Vdash \NE^+$, then there is a non-empty finite set of points $S_x$ such that \vspace{-.5ex}
    \begin{itemize}
        \item $x=\bigcurlyvee S_x$,
        \item each $y\in S_x$ defines a valuation, i.e. $\forall p\in\mathsf{P}:y\Vdash p^+ \text{ iff } y\nVdash p^-$,
        \item $G(\mathbb{M}, x)=\{v_{y}\mid y\in S_x\}$.
    \end{itemize}
    Moreover, whenever $S_x'$ is a non-empty finite set satisfying the first two items, it also satisfies the third.
\end{lemma}
\begin{proof}
    We first prove existence, and then uniqueness up to the set of valuations defined (the `Moreover'-claim).
    
    For the existence claim, enumerate the proposition letters $\{p_1,\hdots, p_n\}=\mathsf{P}$. 
    By induction, we show that for all $0\leq i\leq n$, there is a finite set $S_i$ of elements such that (a) $\bigcurlyvee S_i=x$ and (b) for every $y\in S_i$, \vspace{-.5ex}
    \[
        \text{for all }1\leq j\leq i{:}\; y\Vdash p_j^+ \text{ iff } y\nVdash p_j^-.
    \]
    This shows that $S_x\mathrel{:=}S_n$ witnesses the first two items of the lemma; the last item will be proven by a subsequent argument.
    \\
    For the induction start, set $S_0\mathrel{:=}\{x\}$. 
    \\
    For the induction step, assume the induction hypothesis holds for some $0\leq i<n$. We show that it holds for $i+1$. Let $y\in S_i$ be arbitrary. It is then enough to show that there is a finite set $S_{i+1, y}$ s.t. $\bigcurlyvee S_{i+1, y}=y$ and for every $z\in S_{i+1, y}$, \vspace{-.5ex}
    \[
        \text{for all }1\leq j\leq i+1{:}\; z\Vdash p_j^+ \text{ iff } z\nVdash p_j^-.
    \]
    \vspace{-3.5ex}
    
    \noindent If $y\Vdash p_{i+1}^+ \text{ or } y\Vdash p_{i+1}^-$, set $S_{i+1, y}\mathrel{:=}\{y\}$. This works by the induction hypothesis: the \textit{exclusive} or for $p_{i+1}$ -- that is, $[y\Vdash p_{i+1}^+ \text{ iff } y\nVdash p_{i+1}^-]$ -- follows from $y\Vdash \NE^+$, which, in turn, follows from $x\Vdash \Gamma$ and $[y\Vdash p_{i}^+ \text{ iff } y\nVdash p_{i}^-]$ (or from $x\Vdash\NE^+$).
    \\
    If not, then employing Lemma~\ref{lm:constructionv}, we find $y^+\curlyvee y^-=y$ with $y^\pm \Vdash p_{i+1}^\pm$ and $y^\pm \nVdash p_{i+1}^\mp$. Thus, setting $S_{i+1, y}\mathrel{:=}\{y^+, y^-\}$ works, as $y^\pm \nVdash p_{i+1}^\mp$ implies $y^\pm \Vdash \NE^+$; and $y^\pm\leq y$ entails, via the induction hypothesis and $y^\pm \Vdash \NE^+$, that $\text{for all }1\leq j\leq i{:}\; y^\pm\Vdash p_j^+ \text{ iff } y^\pm\nVdash p_j^-$. This completes the induction.

    It remains to establish the third item, as well as the subsequent uniqueness claim; since our proof of the former applies verbatim to the latter, we only include explicit proof of the former. Clearly, $G(\mathbb{M}, x)\supseteq \{v_{x'}\mid x'\in S_n\}$. Conversely, suppose $v\in G(\mathbb{M}, x)$. Then, by definition of $G$, there is $y\leq x=\bigcurlyvee S_n$ such that $v_y=v$. Note that $y\nVdash \NE^-$. By distributivity, for each $x'\in S_n$, there is $x'_y\leq x'$ such that
    \[
        y=\bigcurlyvee_{x'\in S_n}x'_y.
    \]
    (The generic case $S_n=\{x\}$ too is handled by the forthcoming type of argument.) Consequently, there must be $x'\in S_n$ such that
    \[
        x'_y\nVdash \NE^-,
    \]
    since $V(\NE^-)$ is 
    closed under (non-empty) finite joins because $V(p^\pm)$ are and $V(\NE^-)=\bigcap_{p\in\mathsf{P}}(V(p^+)\cap V(p^-))$. But then $x'_y\Vdash \NE^+$, so as $x'_y\leq y$, we have that for all $p\in\mathsf{P}$,
    \[
        x'_y\Vdash p^+ \Leftrightarrow y\Vdash p^+ \quad\text{and}\quad x'_y\Vdash p^- \Leftrightarrow y\Vdash p^-.
    \]
    By the same argument, $x'_y\leq x'\in S_n$ implies that
    \[
        x'_y\Vdash p^+ \Leftrightarrow x'\Vdash p^+ \quad\text{and}\quad x'_y\Vdash p^- \Leftrightarrow x'\Vdash p^-.
    \]
    Thus, $x'$ defines the same valuation as $y$, so $v=v_y=v_{x'}\in \{v_{x'}\mid x'\in S_n\}$, as desired.
\end{proof}

\subsection{One auxiliary lemma for fullness}
\begin{lemma}[Restatement of Lemma~\ref{lm:adjF}, with proof]\label{lm:adjFproof}
    For all teams $t$, $F(t)\Vdash \Gamma$ and $t=GF(t)$.
\end{lemma}
\begin{proof}
    Recall from Definition~\ref{def:adjF} that $F(t)=((\mathcal{P}(t), \cup,V),t)$. By Lemma~\ref{lm:closurevalideals}, we have that for all $s\subseteq t$ and $p\in \mathsf{P}$,
    \begin{align*}
        s\vDash p\qquad \overset{\text{(def)}}{\text{iff}}\qquad (\mathcal{P}(t), \cup,V),s\Vdash p_+\qquad \text{iff}\qquad (\mathcal{P}(t), \cup,V),s\Vdash p^+,\\
        s\Dashv p\qquad \overset{\text{(def)}}{\text{iff}}\qquad (\mathcal{P}(t), \cup,V),s\Vdash p_-\qquad \text{iff}\qquad (\mathcal{P}(t), \cup,V),s\Vdash p^-.
    \end{align*}
    But then $s\subseteq t$ defines a classical valuation in $(\mathcal{P}(t), \cup,V)$  iff $s=\{v\}$ is a singleton, and in which case it defines $v$. Thus we have $t=GF(t)$, at least provided that $F$ is well-defined to begin with. For well-definedness, we need to show that $F(t)\Vdash \Gamma$, i.e.
    
    \vspace{1mm}
    
    \[
        (\mathcal{P}(t), \cup,V),t\Vdash \PastH(\NE^+\lor \NE^-),\qquad (\mathcal{P}(t), \cup,V),t\Vdash \bigwedge_{p\in \mathsf{P}}\Sup p^+p^-.
    \]
    The latter is a consequence of the above because $t\vDash p\lor \neg p$ and $(p\lor \neg p)^+=\Sup p^+p^-$. For the former, observe first that for all $p\in\mathsf{P}$ and $s\subseteq t$, 
    \[
        s\vDash \NE\quad \text{iff}\quad s\neq\varnothing\quad \text{iff}\quad \text{$s\nvDash p$ or $s\nDashv p$} \quad \text{iff} \quad \text{$s\nVdash p^+$ or $s\nVdash p^-$}\quad \text{iff}\quad s\Vdash {\sim}(p^+\land p^-).
    \]
    Similarly, observe that
    \[
        s\Dashv \NE\quad \text{iff}\quad s=\varnothing\quad \text{iff}\quad \text{$s\vDash p$ and $s\Dashv p$} \quad \text{iff} \quad \text{$s\Vdash p^+$ and $s\Vdash p^-$}\quad \text{iff}\quad s\Vdash (p^+\land p^-).
    \]
    Using this, it therefore follows that $t\Vdash \PastH(\NE^+\lor \NE^-)$, since for all $s\subseteq t$, we have 
    \[  
        s\Vdash \NE^+\lor \NE^-\quad \text{iff}\quad s\Vdash \NE^+ \text{ or }s\Vdash \NE^-\quad \text{iff}\quad s\vDash \NE\text{ or }s\Dashv \NE\quad \text{iff}\quad s\neq\varnothing\text{ or }s=\varnothing.
    \]
\end{proof}

\vspace{-2.5ex}

\subsection{Excursion: $G$ and Birkhoff's representation theorem}\label{app:Birkhoff} 

In the main text, we indicated how our map $G$ resembles Birkhoff's map. 
To bring this out more clearly, consider the $\mathsf{K3}$-version of $\mathsf{BSL}$, where teams are instead sets of \emph{partial} functions $v:\mathsf{P}\to\{0,1\}$. As noted earlier, with minor tweaks to $\Gamma$, our translation goes through and an analogue of Theorem~\ref{th:translation} obtains. The proof remains largely the same, but with one notable subtlety: $(\mathbb{M},x)$ is not mapped to the set of \textit{all} partial valuations defined by points below $x$, but only those that are \textit{join-irreducibly defined}. 

To explain this latter notion, say that points $x,x'\in\mathbb{M}$ are \emph{$\mathsf{P}$-equivalent} just in case they agree on $p^+$ and $p^-$ for all $p\in \mathsf{P}$, i.e.
\[
            x\Vdash p^+ \Leftrightarrow x'\Vdash p^+ \quad\text{and}\quad x\Vdash p^- \Leftrightarrow x'\Vdash p^-.
        \]
As before, say that a point $x$ defines a partial valuation $v:\mathsf{P}\to\{0,1\}$ if $\forall p\in\mathsf{P}(x\Vdash p^+ \Leftrightarrow v(p)=1, \; x\Vdash p^- \Leftrightarrow v(p)=0)$. But say now that $x$ defines it \emph{join-irreducibly} if, in addition, whenever $x=\bigcurlyvee X$ for a non-empty finite set $X$, $x$ is $\mathsf{P}$-equivalent to some $x'\in X$.

So to summarize and make things explicit, the definition of $G$ (or call it $G_{\mathsf{K3}}$) comes down to:
    \[
        G_{\mathsf{K3}}(\mathbb{M}, x)\mathrel{:=}\{v_{y}\mid y\leq x\text{ and $y$ join-irreducibly defines some valuation}\},
    \]
    rendering the analogy with Birkhoff's representation apparent.

\printbibliography\label{biblio}

\end{document}